\documentclass{amsart}
\usepackage{graphicx}
\usepackage{amsmath}
\usepackage{amsthm}
\usepackage{amsfonts}
\usepackage{amssymb, amscd}

\usepackage{color}

\usepackage[all]{xy}

\newtheorem{theorem}{Theorem}[section]
\newtheorem{corollary}[theorem]{Corollary}
\newtheorem{lemma}[theorem]{Lemma}
\newtheorem{proposition}[theorem]{Proposition}
\newtheorem{example}[theorem]{Example}
\theoremstyle{definition}
\newtheorem{definition}[theorem]{Definition}
\theoremstyle{remark}
\newtheorem{remark}[theorem]{Remark}
\numberwithin{equation}{section}

\newcommand{\K}{\mathbb K}
\newcommand{\Z}{\mathbb Z}

\newcommand{\A}{\mathcal{A}}
\newcommand{\B}{\mathfrak{A}}

\newcommand{\LL}{\mathfrak{L}}
\newcommand{\g}{\mathfrak{g}}

\newcommand{\ad}{\emph{ad}}

\begin{document}

\title[  Hom-Lie Algebras  with Symmetric Invariant
NonDegenerate
 Bilinear Forms]
{  Hom-Lie Algebras  with Symmetric Invariant
NonDegenerate
 Bilinear Forms}%
\author{Sa\"{\i}d BENAYADI  and Abdenacer MAKHLOUF  }%
\address{Said Benayadi, Universit\'{e} Paul Verlaine-Metz, LMAM CNRS-UMR 7122, Ile de Saulcy F-57045 Metz-cedex 1, France}
\email{benayadi@univ-metz.fr}
\address{Abdenacer Makhlouf, Universit\'{e} de Haute Alsace,  Laboratoire de Math\'{e}matiques, Informatique et Applications,
4, rue des Fr\`{e}res Lumi\`{e}re F-68093 Mulhouse, France}%
\email{Abdenacer.Makhlouf@uha.fr}

\thanks {
}

%
\begin{abstract}
The aim of this paper is to introduce and study quadratic Hom-Lie algebras, which are Hom-Lie algebras with symmetric invariant nondegenerate
 bilinear forms. We provide several constructions leading to examples and  extend the double extension theory to Hom-Lie algebras. We reduce the case where the twist map is invertible to the study of involutive quadratic Lie algebras. We establish a correspondence between the class of involutive quadratic Hom-Lie algebras and  quadratic simple Lie algebras with symmetric involution. Centerless involutive quadratic Hom-Lie algebras are characterized. Also elements of  a representation theory for
Hom-Lie algebras, including adjoint and coadjoint representations are supplied with application to quadratic Hom-Lie algebras.
\end{abstract}
\maketitle

\section*{Introduction}

 The Hom-algebra structures
arose first in quasi-deformation of Lie algebras of vector fields.
Discrete modifications of vector fields via twisted derivations lead
to Hom-Lie and quasi-Hom-Lie structures in which the Jacobi
condition is twisted. The first examples of $q$-deformations, in which the derivations are replaced by $\sigma$-derivations,
concerned the Witt and Virasoro algebras, see for example
\cite{AizawaSaito,ChaiElinPop,ChaiKuLukPopPresn,ChaiIsKuLuk,ChaiPopPres,CurtrZachos1,
DaskaloyannisGendefVir,Kassel1, LiuKeQin,Hu}. A  general study and
construction of Hom-Lie algebras are considered in
\cite{HLS,LS1,LS2} and a more general framework bordering color and
super Lie algebras was introduced
 in \cite{HLS,LS1,LS2,LS3}. In the subclass of Hom-Lie
algebras skew-symmetry is untwisted, whereas the Jacobi identity is
twisted by a single linear map and contains three terms as in Lie
algebras, reducing to ordinary Lie algebras when the twisting linear
map is the identity map.

The notion of Hom-associative algebras generalizing associative
algebras to a situation where associativity law is twisted by a
linear map was introduced  in \cite{MS}, it turns out  that the
commutator bracket multiplication defined using the multiplication
in a Hom-associative algebra leads naturally to Hom-Lie algebras.
This provided a different  way of constructing Hom-Lie algebras. The
Hom-Lie-admissible algebras and more general $G$-Hom-associative
algebras with subclasses of Hom-Vinberg and pre-Hom-Lie algebras,
generalizing to the twisted situation Lie-admissible algebras,
$G$-associative algebras, Vinberg and pre-Lie algebras respectively,
and shown that for these classes of algebras the operation of taking
commutator leads to Hom-Lie algebras as well. The enveloping
algebras of Hom-Lie algebras were discussed in \cite{Yau:EnvLieAlg}.
The fundamentals of the formal deformation theory and associated
cohomology structures for Hom-Lie algebras have been considered
initially  in \cite{HomDeform} and completed in \cite{AEM}. Simultaneously, in \cite{Yau:HomolHomLie}
elements of homology for Hom-Lie algebras  have been developed.
In \cite{HomHopf} and \cite{HomAlgHomCoalg}, the theory of
Hom-coalgebras and related structures are developed. Further development could be found in \cite{AmmarMakhloufJA2010,AM2008,Canepl2009,JinLi,Yau:HomBial}.

The quadratic Lie algebras, also  called metrizable or orthogonal,  are intensively studied, one of the  fundamental results
of constructing and characterizing  quadratic Lie algebras  is due to Medina
and Revoy (see \cite{MedinaRevoy}) using double extension,  while the concept of  $T^*$-extension  is due
to Bordemann (see \cite{Bordemann97}).  The $T^*$-extension concerns nonassociative algebras with nondegenerate associative symmetric bilinear form, such algebras are called  metrizable algebas. In \cite{Bordemann97}, the metrizable  nilpotent associative algebras and metrizable  solvable Lie algebras are described.  The study of graded quadratic Lie algebras could be found in \cite{BenamorBenayadi} .

The purpose of this paper is  to study and construct quadratic Hom-Lie algebras. In  the
first Section  we summarize the definitions and some key constructions of Hom-Lie algebras. Section 2 is dedicated to a theory of representations of Hom-Lie algebras including adjoint and coadjoint representations. In Section 3 we introduce the notion of quadratic Hom-Lie algebra and give some properties. Several procedures of construction leading to many examples are provided in Section 4. We show in Section 5  that there exists biunivoque correspondence between some classes of Lie algebras and classes of Hom-Lie algebras. In Section 6 we study simple and semisimple involutive Hom-Lie algebras. The last  Section aims to extend  double extension theory to Hom-Lie algebras. We give a characterization of a class of quadratic Hom-Lie algebras. Mainly we a give a structure theorem of centerless involutive quadratic Hom-Lie algebra.

\section{Preliminaries} \label{sect1}
 In the following  we
summarize the definitions of Hom-Lie and
Hom-associative algebraic structures  (see \cite{MS}) generalizing
the well known Lie and associative algebras. Also we define
the notion of modules over Hom-algebras.

Throughout the article we let  $\mathbb{K}$ be an algebraically
closed field of characteristic $0$.  We mean by a Hom-algebra a triple $(A,\mu,\alpha )$ consisting of a vector space $A$, a bilinear map $\mu$ and  a linear map $\alpha$. In all the examples involving  the unspecified products are either given by skewsymmetry or equal to zero.

\subsection{Definitions}
The  notion of Hom-Lie algebra  was introduced by Hartwig, Larsson and
Silvestrov in \cite{HLS,LS1,LS2} motivated initially by examples of deformed
Lie algebras coming from twisted discretizations of vector fields.
In this article, we follow notations and a slightly more general definition
of Hom-Lie algebras from \cite{MS}.

\begin{definition} \label{def:HomLie}
A \emph{Hom-Lie algebra} is a triple $(\g, [\ ,
\ ], \alpha)$ consisting of a linear space $\g$ on which  $[\ , \ ]: \g\times \g \rightarrow \g$ is
a bilinear map and $\alpha: \g \rightarrow \g$
 a linear map
 satisfying
\begin{eqnarray} & [x,y]=-[y,x],
\quad {\text{(skew-symmetry)}} \\ \label{HomJacobiCondition} &
\circlearrowleft_{x,y,z}{[\alpha(x),[y,z]]}=0 \quad
{\text{(Hom-Jacobi condition)}}
\end{eqnarray}
for all $x, y, z$ from $\g$, where $\circlearrowleft_{x,y,z}$ denotes
summation over the cyclic permutation on $x,y,z$.
\end{definition}
We recover classical Lie algebra when $\alpha =id_\g$ and the identity \eqref{HomJacobiCondition} is the Jacobi identity in this case.

Let $\left( \g,\mu,\alpha \right) $ and
$\mathfrak{g}^{\prime }=\left( \g^{\prime },\mu ^{\prime
},\alpha^{\prime }\right) $ be two Hom-Lie algebras. A linear map
$f\ :\g\rightarrow \g^{\prime }$ is
a \emph{morphism of Hom-Lie algebras} if%
$$
\mu ^{\prime }\circ (f\otimes f)=f\circ \mu \quad \text{
and } \qquad f\circ \alpha=\alpha^{\prime }\circ f.
$$

In particular, Hom-Lie algebras $\left( \g,\mu,\alpha \right) $ and
$\left( \g,\mu ^{\prime },\alpha^{\prime }\right) $ are isomorphic if
there exists a
bijective linear map $f\ $such that%
$$
\mu =f^{-1}\circ \mu ^{\prime }\circ (f\otimes f)\qquad
\text{ and }\qquad \alpha= f^{-1}\circ \alpha^{\prime }\circ
f.
$$

A subspace  $I$ of $\g$ is said to be an \emph{\emph{ideal}} if for $x\in I$ and $y\in \g$ we have $[x,y]\in I$ and $\alpha (x)\in I.$ A Hom-Lie algebra in which the commutator is not identically zero and which has no proper ideals is called simple.

\begin{example}
Let $\{x_1,x_2,x_3\}$  be a basis of a $3$-dimensional linear space
$\g$ over $\K$. The following bracket and   linear map $\alpha$ on
$\g=\K^3$ define a Hom-Lie algebra over $\K${\rm :}
$$
\begin{array}{cc}
\begin{array}{ccc}
 [ x_1, x_2 ] &= &a x_1 +b x_3 \\ {}
 [x_1, x_3 ]&=& c x_2  \\ {}
 [ x_2,x_3 ] & = & d x_1+2 a x_3,
 \end{array}
 & \quad
  \begin{array}{ccc}
  \alpha (x_1)&=&x_1 \\
 \alpha (x_2)&=&2 x_2 \\
   \alpha (x_3)&=&2 x_3
  \end{array}
\end{array}
$$
with $[ x_2, x_1 ]$, $[x_3, x_1 ]$ and  $[
x_3,x_2 ]$ defined via skewsymmetry. It is not a
Lie algebra if and only if $a\neq0$ and $c\neq0$,
since
$$[x_1,[x_2,x_3]]+[x_3,[x_1,x_2]]
+[x_2,[x_3,x_1]]= a c x_2.$$
\end{example}
\begin{example}[Jackson $\mathfrak{sl}_2$]
The Jackson $\mathfrak{sl}_2$ is a $q$-deformation of the classical  $\mathfrak{sl}_2$. It carries a  Hom-Lie algebra structure but not a Lie algebra structure. It is defined with respect to a basis $\{x_1,x_2,x_3\}$  by the brackets and a linear  map $\alpha$ such that
$$
\begin{array}{cc}
\begin{array}{ccc}
 [ x_1, x_2 ] &= &-2q x_2 \\ {}
 [x_1, x_3 ]&=& 2 x_3 \\ {}
 [ x_2,x_3 ] & = & - \frac{1}{2}(1+q)x_1,
 \end{array}
 & \quad
  \begin{array}{ccc}
  \alpha (x_1)&=&q x_1 \\
 \alpha (x_2)&=&q^2  x_2 \\
   \alpha (x_3)&=&q x_3
  \end{array}
\end{array}
$$
where $q$ is a parameter in $\K$. if $q=1$ we recover the classical $\mathfrak{sl}_2$.
\end{example}
For simplicity we will use in the sequel the following terminology and notations.
\begin{definition}\label{defMultRegInvo}
Let  $(\g, [\cdot, \cdot], \alpha)$ be a Hom-Lie algebra. The Hom-algebra is called
\begin{itemize}
\item \emph{multiplicative Hom-Lie algebra} if  $\forall x,y\in\g$ we have $\alpha([x,y])=[\alpha (x),\alpha (y)]$;
\item  \emph{regular Hom-Lie algebra} if $\alpha$ is an automorphism;
 \item \emph{involutive Hom-Lie algebra} if $\alpha$ is an involution, that is $\alpha^2=id$.
\end{itemize}
The \emph{center} of the Hom-Lie algebra is denoted $\mathcal{Z}(\g)$ and defined by
$$\mathcal{Z}(\g)=\{x\in\g : [x,y]=0\ \forall y\in\g \}.
$$

\end{definition}


We recall in the following the definition of Hom-associative
algebra which provide a different way for constructing Hom-Lie algebras by extending the
fundamental construction of Lie algebras from associative algebras via commutator bracket multiplication. This structure was introduced
by the second author and Silvestrov (see \cite{MS}).
\begin{definition}
A \emph{Hom-associative algebra}  is a triple $( A, \mu,
\alpha) $ consisting of a linear space  $A$,  $\mu :A\times A\rightarrow  A$ is a bilinear map
and $\alpha: A \rightarrow A$ is a linear map, satisfying
\begin{equation}\label{Hom-ass}
\mu(\alpha(x), \mu (y, z))= \mu (\mu (x, y), \alpha (z)).
\end{equation}
\end{definition}

There is a functor from the category of Hom-associative algebras in
the category of Hom-Lie algebras.
\begin{proposition}[\cite{MS}]
Let $( A, \mu, \alpha) $ be a Hom-associative algebra defined on the linear space $A$  by the
multiplication $\mu$ and a homomorphism $\alpha$. Then the triple $( A, [~,~], \alpha) $, where the bracket
is defined for  $x,y \in A$ by  $ [ x,y ]=\mu (x,y)-\mu (y,x )
$, is a Hom-Lie algebra.
\end{proposition}

A structure of module over Hom-associative algebras is defined in  \cite{HomHopf} and \cite{HomAlgHomCoalg} as
follows.
\begin{definition}
Let $(\A,\mu,\alpha)$ be a Hom-associative
 algebra. A   (left) $\A$-module is a triple
$(M,f,\gamma)$ where $M$ is a $\K$-vector space and $f,\gamma$ are
$\K$-linear maps, $f:  M \rightarrow M$ and $\gamma : \A \otimes M
\rightarrow M$, such that the following diagram commutes:

$$
\begin{array}{ccc}
\A\otimes \A\otimes M & \stackrel{\mu \otimes f}{\longrightarrow } &
\A\otimes
M \\
\ \quad \left\downarrow ^{\alpha \otimes \gamma }\right. &  & \quad
\left\downarrow
^\gamma \right. \\
\A\otimes M & \stackrel{\gamma }{\longrightarrow } & M
\end{array}
$$
\end{definition}
\begin{remark}
A Hom-associative algebra $(\A,\mu,\alpha)$ is  a left
$\A$-module with $M=\A$, $f=\alpha$ and $\gamma =\mu$.
\end{remark}

The following result shows that Lie algebras deform into Hom-Lie
algebras via endomorphisms.

\begin{theorem}[\cite{Yau:HomolHomLie}]\label{thmYauConstrHomLie}
Let $(\g,[~,~])$ be a Lie algebra and $\alpha :
\g\rightarrow \g$ be a Lie algebra endomorphism. Then
$\mathfrak{g}_\alpha =(\g,[~,~]_\alpha,\alpha)$ is a Hom-Lie algebra,
where $[~,~]_\alpha=\alpha\circ[~,~]$.

Moreover, suppose that  $(\g',[~,~]')$ is another Lie
algebra and  $\alpha ' : \g'\rightarrow \g'$ is a Lie algebra
endomorphism. If $f:\g\rightarrow \g'$ is a Lie algebra morphism that
satisfies $f\circ\alpha=\alpha'\circ f$ then
$$f:(\g,[~,~]_\alpha,\alpha)\longrightarrow (\g',[~,~]'_{\alpha '},\alpha ')
$$
is a morphism of Hom-Lie algebras.
\end{theorem}
\begin{proof}
Observe that $[\alpha (x),[y,z]_\alpha]_\alpha =\alpha [\alpha
(x),\alpha[y,z]]=\alpha^2 [x,[y,z]] $. Therefore the Hom-Jacobi
identity for $\mathfrak{g}_\alpha=(\g,[~,~]_\alpha,\alpha)$ follows
obviously from the Jacobi identity of $(\g,[~,~])$. The
skew-symmetry and the second assertion are proved similarly.
\end{proof}

In the sequel we denote by $\g_\alpha$ the Hom-Lie algebra $(\g,\alpha\circ [~,~],\alpha)$ corresponding to a given Lie algebra $(\g, [\ ,\ ] )$ and an endomorphism $\alpha$. We say that the Hom-Lie algebra is obtained by composition.

Let $(\g,[~,~],\alpha)$ be a regular  Hom-Lie algebra. It was observed in \cite{Gohr} that the composition method using $\alpha^{-1}$ leads to a Lie algebra.

\begin{proposition}\label{CompoMethodLie}
Let $(\g,[\ ,\ ] ,\alpha)$ be a regular  Hom-Lie algebra.  Then  $(\g,[~,~]_{\alpha^{-1}}=\alpha^{-1}\circ [~,~])$ is a Lie algebra.
\end{proposition}
\begin{proof}It follows from
\begin{align*}
\circlearrowleft_{x,y,z}{[x,[y ,z]_{\alpha^{-1}}]_{\alpha^{-1}}}=\circlearrowleft_{x,y,z}{\alpha^{-1}( [x,\alpha^{-1} ([y,z] ) ])}=\circlearrowleft_{x,y,z}{{\alpha^{-2}}[\alpha (x),[y,z]]}=0.
\end{align*}
\end{proof}
\begin{remark}
In particular  the proposition is valid  when $\alpha$ is an involution.
\end{remark}
We may also derive new Hom-Lie algebras from a given multiplicative  Hom-Lie algebra using the following procedure.

\begin{definition}[\cite{Yau4}]\label{defiNDerivedHom} Let  $\left( \g,[\ ,\ ] ,\alpha \right) $ be a multiplicative Hom-Lie algebra and $n\geq 0$. The  $n$th derived Hom-algebra  of $\g$ is  defined by
\begin{equation}\label{DerivedHomAlgtype1}
\g _{(n)}=\left( \g,[\ ,\ ]^{(n)}=\alpha^{n }\circ[\ ,\ ] ,\alpha^{n+1} \right),
\end{equation}
Note that $\g_{(0)}=\g$ and  $\g_{(1)}=\left( \g,[\ ,\ ]^{(1)}=\alpha\circ[\ ,\ ] ,\alpha^{2} \right)$.
\end{definition}
Observe that for $n\geq 1$ and $x,y,z\in \g$ we have
\begin{eqnarray*}
[[x,y]^{(n)},\alpha^{n+1}(z)]^{(n)}&=& \alpha^{n }([\alpha^{n }([x,y]),\alpha^{n+1}(z)])\\
\ &=& \alpha^{2n }([[x,y],\alpha(z)]).
\end{eqnarray*}
Hence, one obtains the following result.
\begin{theorem}[\cite{Yau4}]\label{ThmConstrNthDerivedLie}
Let   $\left( \g,[\ ,\ ] ,\alpha \right) $ be a multiplicative Hom-Lie algebra. Then its  $n$th derived Hom-algebra  is  a Hom-Lie algebra.
\end{theorem}

In the following we construct  Hom-Lie algebras involving elements of the centroid of Lie algebras. Let   $(\mathfrak{g}, [\cdot, \cdot])$ be a Lie algebra. An endomorphism  $\theta\in End (\g) $ is said to be an element of the centroid if $\theta[x,y]=[\theta(x),y]$  for any $x,y\in\g$. The centroid is defined by
$$Cent (\g )=\{ \theta\in End (\g) : \theta [x,y]=[\theta (x),y], \  \forall x,y\in\g\}.
$$
The same definition is assumed for Hom-Lie algebra.
\begin{proposition}
Let   $(\mathfrak{g}, [\cdot, \cdot])$ be a Lie algebra and  $\theta\in Cent (\g) $. Set for $x,y\in\g$
\begin{align*}
\{x,y\}&=[\theta (x),y],\\
[x,y]_\theta &=[\theta (x),\theta (y)].
\end{align*}
Then $(\mathfrak{g},\{\cdot, \cdot\},\theta )$ and $(\mathfrak{g},[\cdot, \cdot]_\theta,\theta )$ are Hom-Lie algebras.
\end{proposition}
\begin{proof}
For  $\theta\in Cent (\g) $ we have $[\theta(x),y]=\theta([x,y])=-\theta([y,x])=-[\theta(y),x]=[x,\theta(y)]$. Then
$$\{x,y\}=[\theta (x),y]=-[\theta (y),x]=-\theta[y,x]=-\{y,x\}.$$
Also we have
\begin{align*}
\{ \theta (x),\{y,z\}\}&= [\theta^2 (x),\{y,z\}]=[\theta^2 (x),[\theta (y),z]]\\
\ &= \theta([\theta (x),[\theta (y),z]])=[\theta (x),\theta([\theta (y),z])]\\
\ &= [\theta (x),[\theta (y),\theta(z)]].
\end{align*}
It follows $\circlearrowleft_{x,y,z}{\{\theta (x),\{y,z\}\}}=\circlearrowleft_{x,y,z}{[\theta (x),[\theta (y),\theta (z)]]}=0$ since $(\g,[\ ,\ ])$ is a Lie algebra.
Therefore the Hom-Jacobi is satisfied. Thus  $(\mathfrak{g},\{\cdot, \cdot\},\theta )$ is a Hom-Lie algebra.

Similarly we have the skewsymmetry and the Hom-Jacobi identity  satisfied for   $(\mathfrak{g},[\cdot, \cdot]_\theta,\theta )$. Indeed
$$[x,y]_\theta=[\theta (x),\theta (y)]=-[\theta (y),\theta (x)]=-[y,x]_\theta.$$
and
\begin{align*}
[\theta (x),[y ,z]_\theta]_\theta&= [\theta^2 (x),\theta([y,z]_\theta ) ]=[\theta^2 (x),\theta ([\theta (y),\theta (z)]]
= \theta^2([ \theta(x),[\theta(y),\theta(z)]].
\end{align*}
which leads to  $\circlearrowleft_{x,y,z}{[\theta (x),[y,z]_\theta]_\theta}=\theta^2(\circlearrowleft_{x,y,z}{[\theta (x),[\theta (y),\theta (z)]]})=0$.
\end{proof}

\section{Representations of Hom-Lie algebras } \label{sect3}

In this section we introduce a representation theory of Hom-Lie algebras
and discuss the cases of adjoint and coadjoint representations for Hom-Lie
algebras. The representations of Hom-Lie algebras were considered independently in a general framework in \cite{Sheng}. Moreover the author give the corresponding coboundary operator, the cohomologies associated to the adjoint representations and the trivial representation are explicitly defined. Notice that the coadjoint representations are not discussed there.
\begin{definition}
Let $(\g, [\ , \ ], \alpha)$ be a Hom-Lie algebra.
A \emph{representation} of $\mathfrak{g}$ is a triple $(V,\rho,\beta)$,
where $V$ is a $\K$-vector space, $\beta \in End(V)$ and $\rho :
\g\rightarrow End(V)$ is a linear map satisfying
\begin{equation}\label{representation}
\rho ([x,y])\circ \beta=\rho(\alpha (x))\circ\rho (y)-\rho(\alpha
(y))\circ\rho (x) \quad \forall x,y\in \g
\end{equation}
\end{definition}

One recovers the definition of a representation in the case of  Lie
algebras by setting $\alpha =Id_\g$ and $\beta=Id_V$.

\begin{definition}
Let $(\g,[\ , \ ], \alpha)$ be a Hom-Lie algebra. Two representations $(V,\rho,\beta)$ and $(V',\rho',\beta')$  of $\mathfrak{g}$ are said to be isomorphic if there exists a linear map
$\phi\ :V \rightarrow V^{\prime }$ such that
$$
 \forall x\in \g \ \ \rho ^{\prime }(x)\circ \phi=\phi\circ \rho(x) \quad \text{
and } \qquad \phi\circ \beta=\beta^{\prime }\circ \phi.
$$
\end{definition}

in the following, we discuss some properties of Hom-Lie algebras representations.

\begin{proposition}\label{prop2}
Let $(\g,[\ , \ ]_\mathfrak{g}, \alpha)$ be a
Hom-Lie algebra and  $(V,\rho,\beta)$ be a representation of
$\mathfrak{g}$.

The direct summand $\g\oplus V$ with a bracket defined by
\begin{equation}
[x+u,y+w]:=[x,y]_\mathfrak{g}+\rho(x)(w)-\rho (y)(u)  \quad \forall
x,y\in \g \ \forall u,w\in V
\end{equation}
and the twisted map  $\gamma : \g\oplus V \rightarrow \g\oplus V$
defined by
\begin{equation}\gamma (x+w)=\alpha (x) +\beta (u)\quad
\forall x\in \g \ \forall u\in V.
\end{equation}
is  a Hom-Lie algebra.
\end{proposition}

\begin{proof}
The skew-symmetry of the bracket is obvious. We show that the
Hom-Jacobi identity is satisfied:

Let $ x,y,z\in \g$ and $\forall u,v,w\in V.$
\begin{eqnarray*}
\circlearrowleft_{(x,u),(y,v),(z,w)}[\gamma
(x+u),[y+v,z+w]]=\circlearrowleft_{(x,u),(y,v),(z,w)}[\alpha
(x)+\beta (u),[y,z]_\mathfrak{g}+\rho (y)(w)-\rho (z)(v)]\\
=\circlearrowleft_{(x,u),(y,v),(z,w)}[\alpha
(x),[y,z]_\mathfrak{g}]_\mathfrak{g}+\rho (\alpha (x)(\rho
(y)(w)-\rho (z)(v))-\rho ([y,z]_\mathfrak{g})(\beta(u))\\
=\circlearrowleft_{(x,u),(y,v),(z,w)}{\rho (\alpha (x)(\rho
(y)(w))-\rho (\alpha (x)(\rho (z)(v))-\rho (\alpha (y)(\rho
(z)(u))+\rho (\alpha (z)(\rho (y)(u))}\\
  = \rho (\alpha (x)(\rho
(y)(w))-\rho (\alpha (x)(\rho (z)(v))-\rho (\alpha (y)(\rho
(z)(u))+\rho (\alpha (z)(\rho (y)(u))\\
+\rho (\alpha (y)(\rho (z)(u))-\rho (\alpha (y)(\rho (x)(w))-\rho
(\alpha (z)(\rho (x)(v))+\rho (\alpha (x)(\rho (z)(v))
\\
+\rho (\alpha (z)(\rho (x)(v))-\rho (\alpha (z)(\rho (y)(u))-\rho
(\alpha (x)(\rho (y)(w))+\rho (\alpha (y)(\rho (x)(w))\\
=0
\end{eqnarray*}
where $\circlearrowleft_{(x,u),(y,v),(z,w)}$ denotes summation over
the cyclic permutation on $(x,u),(y,v),(z,w)$.
\end{proof}

Now, we discuss the adjoint representations of a Hom-Lie algebra.

\begin{proposition}
Let $(\g, [\ , \ ], \alpha)$ be a Hom-Lie algebra
and $\ad: \g\rightarrow End(\g)$ be an operator defined for $x\in \g$ by
$\ad(x)(y)=[x,y].$ Then  $(\g,\ad,\alpha)$ is a representation of
$\mathfrak{g}$.
\end{proposition}
\begin{proof}
Since $\mathfrak{g}$ is Hom-Lie algebra,
 the Hom-Jacobi condition on $x,y,z\in \g$ is
 $$[\alpha (x),[y,z]]+[\alpha (y),[z,x]]+[\alpha (z),[x,y]]=0
 $$
 and may be written
  $$\ad[x,y](\alpha(z))=\ad(\alpha (x))(\ad(y)(z))-\ad(\alpha (y))(\ad(x)(z))
 $$
 Then the operator $\ad$ satisfies
 $$\ad[x,y]\circ\alpha=\ad(\alpha (x))\circ \ad(y)-\ad(\alpha (y))\circ(\ad(x).
 $$
 Therefore, it determines a representation of the Hom-Lie algebra
 $\mathfrak{g}$.
\end{proof}

We call the representation defined in the previous proposition
\emph{adjoint representation} of the Hom-Lie algebra.

In the following, we explore the dual representations and coadjoint representations of Hom-Lie algebras.

Let $(\g,[\ , \ ], \alpha)$ be a
Hom-Lie algebra and  $(V,\rho,\beta)$ be a representation of
$\mathfrak{g}$. Let $V^*$ be the dual vector space of $V$.
We define a linear map  $\widetilde{\rho} :\mathfrak{g}\rightarrow End(V^*)$
 by $\widetilde{\rho}(x)=-^{t}\rho (x)$.

Let $f\in V^*$, $x,y\in \g$ and $u\in V$. We compute the right hand side of
the identity (\ref{representation})

\begin{eqnarray*}
(\widetilde{\rho}(\alpha (x))\circ\widetilde{\rho} (y)-\widetilde{\rho}(\alpha
(y))\circ\widetilde{\rho} (x))(f)(u)&=&
(\widetilde{\rho}(\alpha (x))(\widetilde{\rho} (y)(f))-
\widetilde{\rho}(\alpha
(y))(\widetilde{\rho} (x)(f)))(u)\\
\ &=&-\widetilde{\rho} (y)(f)(\rho(\alpha (x))(u))+
\widetilde{\rho} (x)(f)(\rho(\alpha(y))(u))\\
\ &=&f(\rho (y)\rho(\alpha (x))(u))-
f(\rho (x)\rho(\alpha(y))(u))\\
\ &=&f(\rho (y)\rho(\alpha (x))-
\rho (x)\rho(\alpha(y))(u)).
\end{eqnarray*}
On the other hand, we set that the twisted map for
$\widetilde{\rho}$ is $\widetilde{\beta}= ^{t}\beta$, then the left
hand side of (\ref{representation}) writes
\begin{eqnarray*}
((\widetilde{\rho} ([x,y]) \widetilde{\beta})(f))(u)&=&
(\widetilde{\rho} ([x,y])(f \circ \beta )(u),\\
\ &=& -f \circ \beta (\rho ([x,y])(u)).
\end{eqnarray*}
Therefore, we have the following proposition:
\begin{proposition}
Let $(\g, [\cdot, \cdot], \alpha)$ be a
Hom-Lie algebra and  $(V,\rho,\beta)$ be a representation of
$\mathfrak{g}$.

The triple $(V^*,\widetilde{\rho},\widetilde{\beta})$, where
$\widetilde{\rho} :\mathfrak{g}\rightarrow End(V^*)$ is
given by $\widetilde{\rho}(x)=-^{t}\rho (x)$, defines a representation
of the Hom-Lie algebra $(\g, [\cdot, \cdot], \alpha)$ if and only if
\begin{equation}
\beta \circ\rho ([x,y])=\rho (x)\rho(\alpha (y))-
\rho (y)\rho(\alpha(x)).
\end{equation}
\end{proposition}
We obtain the following characterization in the case of adjoint representation.
\begin{corollary}
Let $(\g,[\ , \ ], \alpha)$ be a
Hom-Lie algebra and $(\g,\ad,\alpha)$ be the adjoint representation of
 $\mathfrak{g}$, where $\ad: \g\rightarrow End(\g)$.
 We set $\widetilde{\ad}: \g\rightarrow End(\g^*)$ and
 $\widetilde{\ad}(x)(f)=-f\circ\ad(x).$

 Then $(\g^*,\widetilde{\ad},\widetilde{\alpha})$
 is a representation of $\mathfrak{g}$ if and only if
 \begin{equation}
 \alpha ([[x,y],z])=[x,[\alpha(y),z]]-[y,[\alpha (x),z]]\quad \forall x,y,z\in \g.
 \end{equation}
\end{corollary}

\section{ Definition and properties of Quadratic Hom-Lie algebras}\label{sectDef}
 In this section we extend the notion of quadratic Lie algebra  to Hom-Lie algebras and provide some properties.

Let   $(\g,[\ , \ ])$ be a Lie algebra and $B : \g\times \g \rightarrow \K$ a symmetric nondegenerate  bilinear form satisfying
\begin{equation}\label{InvariantForm}
B([x,y],z)=B(x,[y,z]) \quad \forall x, y, z\in \g.
\end{equation}
The identity \eqref{InvariantForm} may be written $ B([x,y],z)=-B(y,[x,z])$ and is called invariance of $B$. The bilinear form $B$ is called, with misuse of language, invariant scalar product.
The triple  $(\mathfrak{g},[\ , \ ],B)$  is called quadratic Lie algebra or sometimes  orthogonal  Lie algebra.

More generally, for nonassociative algebras $(A,\cdot)$, a triple $(A,\cdot,B)$ where $B$ is a symmetric nondegenerate  bilinear form satisfying
\begin{equation}\label{AssoForm}
B(x\cdot y,z)=B(x,y\cdot z) \quad \forall x, y, z\in A
\end{equation}
defines a quadratic algebra, called also metrizable  algebra. A   bilinear form $B$ satisfying \eqref{AssoForm} is said either  associative form or invariant form.

\begin{definition} \label{def:QuadraticHomLie}
Let  $(\g, [\ , \ ], \alpha)$ be a  Hom-Lie algebra and $B : \g\times \g \rightarrow \K$ be an invariant  symmetric nondegenerate bilinear
form satisfying \begin{equation}\label{AlphaComp}
B(\alpha (x), y)=B(x,\alpha (y)) \quad \forall x, y\in \g.
\end{equation} The quadruple $(\g, [\ , \ ], \alpha,B)$ is
called \emph{quadratic} Hom-Lie algebra.

If $\alpha$ is an involution (resp. invertible), the quadratic Hom-Lie algebra is said to be involutive (resp. regular) quadratic Hom-Lie algebra and we write for shortness IQH-Lie algebra (resp. RQH-Lie algebra).
\end{definition}

We recover the classical notion of quadratic Lie algebra when
$\alpha$ is the identity map. One may consider a larger class with a definition without condition \eqref{AlphaComp}. We may also introduce in the following a  generalized quadratic Hom-Lie algebra notion  where the invariance is twisted by a linear map.

\begin{definition} \label{def:GeneralQuadraticHomLie}
A Hom-Lie algebra $(\g, [\ , \ ], \alpha)$  is
called \emph{Hom-quadratic} if there exist a pair $(B,\gamma)$ where
$B : \g\times \g \rightarrow \K$ is a symmetric nondegenerate bilinear
form and $\gamma : \g\rightarrow \g $ is a linear map
 satisfying
\begin{equation}\label{GeneralInvariantForm}
B([x,y],\gamma(z))=-B(\gamma(y),[x,z]) \quad \forall x, y, z\in \g
\end{equation}
\end{definition}
We call the identity \eqref{GeneralInvariantForm} the $\gamma$-invariance of $B$. We recover the quadratic Hom-Lie algebras when $\gamma=id$.


\subsection{Quadratic Hom-Lie algebras and Representation theory}
We establish in the following a connection between quadratic Hom-Lie algebras and representation theory.
\begin{proposition}
Let  $(\g, [\cdot, \cdot], \alpha)$ be a  Hom-Lie algebra. If  there exists  $B : \g\times \g \rightarrow \K$ a bilinear
form such that  the quadruple $(\g, [\cdot, \cdot], \alpha,B)$ is a quadratic Hom-Lie algebra then
\begin{enumerate}
\item  $(\g^*,\widetilde{\ad},\widetilde{\alpha})$
 is a representation of $\mathfrak{g}$
\item The representations $(\g,\ad ,\alpha)$ and $(\g^*,\widetilde{\ad},\widetilde{\alpha})$ are isomorphic.
\end{enumerate}
\end{proposition}
\begin{proof}
To prove the first assertion, we should show that for any $z$ we have
\begin{equation}\label{cond5.2}
\alpha \circ\ad ([x,y])(z)-\rho (x)\ad(\alpha (y))(z)+
\ad (y)\ad(\alpha(x))(z)=0,
\end{equation}
that is
$$
\alpha [[x,y],z]-[x,[\alpha (y),z]]+
[y,[\alpha(x),z]]=0.
$$
Let $u\in \g$
\begin{small}
\begin{align*}
B(\alpha [[x,y],z]-[x,[\alpha (y),z]]+
[y,[\alpha(x),z]],u)=&B(\alpha [[x,y],z],u)-B([x,[\alpha (y),z]],u)+B(
[y,[\alpha(x),z]],u)\\
\ =&  B( [[x,y],z],\alpha (u))+B([\alpha (y),z],[x,u])-B(
[\alpha(x),z],[y,u])
\\
\ =& -(B( z,[[x,y],\alpha (u)])+B(z,[\alpha (y),[x,u]])-B(
z,[\alpha(x),[y,u]]))
 \\
\ =&  -(B( z,[[x,y],\alpha (u)]+[\alpha (y),[x,u]])-[\alpha(x),[y,u]])
\\
\ =&  -(B( z,[\alpha (u),[y,x]])+[\alpha (y),[x,u]])+[\alpha(x),[u,y]]))
\\
\ =& 0.
\end{align*}
\end{small}
The identity \ref{cond5.2} since $B$ is nondegenerate.

For the second assertion we consider the map $\phi:\g \rightarrow \g^\star $ defined by $x \rightarrow B(x,\cdot)$ which is bijective since $B$ is nondegenerate and  prove that it is also a module morphism.
\end{proof}

\subsection{Ideals and Decompositions of Quadratic Hom-Lie algebras}
We introduce the following definitions and give some properties related to ideals of Hom-Lie algebras.
\begin{definition}
Let $(\mathfrak{g}, [\cdot, \cdot], \alpha,B)$ be a quadratic Hom-Lie algebra.
\begin{enumerate}
\item An ideal $I$ of $\g$ is said to be nondegenerate if $B_{|I\times I}$ is nondegenerate.
\item The quadratic Hom-Lie algebra  is said to be  irreducible (or $B$-irreducible)  if $\g$ doesn't contain any nondegenerate ideal $I$ such that $I\neq \{0\} $ and $I\neq \g $.
\item Let $I$ be an ideal of $\g$. The orthogonal $I^\bot$ of $I$ with respect to $B$ is defined by $\{x\in \g : B(x,y)=0 \ \forall y\in I\}.$
    \end{enumerate}
\end{definition}
\begin{remark}
Let  $I$ be a nondegenerate ideal of a quadratic Hom-Lie algebra $(\mathfrak{g}, [\cdot, \cdot], \alpha,B)$.  \\ Then  $(I,[\ ,\ ]_{|I\times I},\alpha_|I,B_{|I\times I})$ is a quadratic Hom-Lie algebra.
\end{remark}
\begin{lemma}
Let  $(\g, [\cdot, \cdot], \alpha)$ be a multiplicative   Hom-Lie algebra. Then the center $\mathcal{Z}(\g)$ is an ideal of $\g.$
\end{lemma}
\begin{proof}
We  have $[\g,\mathcal{Z}(\g)]=\{0\}\subseteq\mathcal{Z}(\g).$ Let $x\in\mathcal{Z}(\g)$ and $y\in\g$. For any $z\in\g$ the invariance and the symmetry of $B$ leads to
$B([\alpha (x),y],z)=B(\alpha(x),[y,z])=B(x,\alpha ([y,z]))=B(x,[\alpha (y),\alpha(z)])=B([x,\alpha (y)],\alpha(z)])=0$ (since $x\in\mathcal{Z}(\g)$).

Then for any $y\in\g$ we have $[\alpha(x),y]=0$ since $B$ is nondegenerate. Thus $\alpha(x)\in\mathcal{Z}(\g)$.
\end{proof}

\begin{lemma}
Let $(\mathfrak{g}, [\cdot, \cdot], \alpha,B)$ be a quadratic Hom-Lie algebra and $I$ be an ideal of $\g$. Then the orthogonal $I^\bot$ of $I$ with respect to $B$   is an ideal of $\g$.
\end{lemma}
\begin{proof}
It is clear that $[\g ,  I^\bot]\subseteq I^\bot$. Let $y\in I $ and $z\in I^\bot$, then $B(\alpha (y),z)=B(y,\alpha (z))=0$ since $\alpha (I)\subseteq I$. We conclude that $I^\bot$ is an ideal  of $\g.$
\end{proof}
\begin{proposition}
Let $(\mathfrak{g}, [\cdot, \cdot], \alpha,B)$ be a quadratic Hom-Lie algebra. Then $\g=\g_1\oplus \cdots \oplus \g_n$ such that
\begin{enumerate}
\item $\g_i$ is an irreducible ideal of $\g$, for any $i\in\{1,\cdots,n\}$,
\item $B(\g_i,\g_j)=\{0\},$ for any $i,j\in\{1,\cdots,n\}$ such that $i\neq j,$
\item $(\mathfrak{g}_i, [\cdot, \cdot]_{|\g_i\times\g_i}, \alpha_{|\g_i},B_{|\g_i\times\g_i})$ is an irreducible quadratic Hom-Lie algebra.
    \end{enumerate}
\end{proposition}
\begin{proof}
By induction on the dimension of $\g.$
\end{proof}

Now, let  $\g=(\g,[ \ ,\ ],\alpha,B)$ be a quadratic multiplicative  Hom-Lie algebra.
We provide in the following  some observations.
\begin{proposition}\label{CentreNul}
If the linear map $\alpha$ is an automorphism and the center $\mathcal{Z}(\g)=\{0\}$ then $\alpha$ is an involution i.e. $\alpha^2=id$.
\end{proposition}
\begin{proof}
Let $x,y,z\in \g$, we have
\begin{eqnarray*}
B([\alpha (x),y],z)&&=B(\alpha (x),[y,z])=B(x,\alpha ([y,z])
\\ \ &&=B(x,[\alpha (y),\alpha (z)])=B([x,\alpha (y)],\alpha (z))
\\ \ &&=B(\alpha ([x,\alpha (y)]),z)=B([\alpha(x),\alpha^2(y)],z).
\end{eqnarray*}
Then $B([\alpha (x),y]-[\alpha(x),\alpha^2(y)],z)=0$ which may be written $B([\alpha (x),y-\alpha^2(y)],z)=0$. Hence, for any $x,y\in\g$ we have $[\alpha(x),(id-\alpha^2)(y)]=0$. Since $\alpha$ is bijective and $\mathcal{Z}(\g)=\{0\}$ then $\alpha^2=id$.
\end{proof}
\begin{proposition}
There exist two nondegenerate  ideals $I,J$  of  $\g=(\g,[ \ ,\ ],\alpha,B)$ such that
\begin{enumerate}
\item $B(I,J)=\{0\}$,
\item $\g=I\oplus J$,
\item $\alpha_{|I}$ is nilpotent and $\alpha_{|J}$ is invertible.
\end{enumerate}
\end{proposition}
\begin{proof}
The fitting decomposition with respect to the linear map $\alpha$ leads to the existence of an integer $n$ such that $\g=I\oplus J$, where $I=Ker( \alpha^n)$ and $J=Im ( \alpha^n)$, satisfying
\begin{itemize}
\item $\alpha (I)\subseteq I$,
\item $\alpha (J)\subseteq I$,
\item $\alpha_{|I}$ is nilpotent,
\item $\alpha_{|J}$ is invertible.
\end{itemize}
Let $x\in \g$, $y\in I$. We have $\alpha^n([x,y])=[\alpha^n(x),\alpha^n(y)]=0$ since $\alpha^n(y)=0$, and  $[x,y]\in I$. Then $[\g,I]\subseteq I$. In addition $\alpha^n (\alpha(y))=\alpha^{n+1}(y)=0$ which implies that $\alpha (y)\in Ker (\alpha^n)$.
Therefore $I$ is an ideal of $\g$.

Let  $x,y\in J$ then there exist $x',y'\in\g$ such that $x=\alpha^n (x')$ and $y=\alpha^n (y')$.  We have $[x,y]=[\alpha^n(x'),\alpha^n(y')]=\alpha^n([x',y'])\in J.$
 In addition $\alpha (J)\subseteq J$. Therefore $J$ is a subalgebra.

 Let $x\in I$ and $y\in J$ . There exists $y'\in\g$ such that  $y=\alpha^n (y')$. For any $z\in\g$, we have
 $B([x,y],z)=-B([y,x],z)=-B(y,[x,z])=-B(\alpha^n (y'),[x,z])=-B(y',\alpha^n ([x,z])=-B(y',[\alpha^n (x),\alpha^n (z)])=0$. Then $[x,y]=0,$ since $B$ is a nondegenerate bilinear form.
 We conclude that $I=Im(\alpha^n)$ is an ideal of $\g$ and $[I,J]=0.$

 Now let $x\in I$ and $y=\alpha^n (y')\in J$, where $y'\in\g$. We have $B(x,y)=B(x,\alpha^n (y'))=B(\alpha^n (x),y')=0$ since $\alpha^n (x)=0$. Therefore $B(I,J)=0.$
\end{proof}

\begin{corollary}
Let $(\mathfrak{g}, [\cdot, \cdot], \alpha,B)$ be a quadratic Hom-Lie algebra which is $B$-irreducible. Then either $\alpha$ is nilpotent or $\alpha$ is an automorphism of $\g$.
\end{corollary}

\section{Construction of Quadratic Hom-Lie algebras and Examples } \label{sect2}
We show in the following
some constructions leading to new  examples of quadratic Hom-Lie algebras. We
use  Theorem \ref{thmYauConstrHomLie} and Theorem \ref{ThmConstrNthDerivedLie} to provide some classes
of quadratic Hom-Lie algebras starting from an ordinary  quadratic
Lie algebras, respectively from any multiplicative quadratic Hom-Lie algebra. Also we provide constructions using elements in the centroid of a Lie algebras and constructions of $T^*$-extension type.

Let $(\g,[~,~],B)$ be a quadratic Lie algebra. We denote by
$Aut_S(\mathfrak{g},B)$ the set of symmetric  automorphisms of
$ \mathfrak{g}$ with respect of $B$, that is automorphisms
$f:\g\rightarrow \g$   such that $B(f(x),y)=B(x,f(y))$, $\forall x,y\in \g.$

\begin{proposition}
Let $(\g,[~,~],B)$ be a quadratic Lie algebra and
$\alpha \in Aut_S(\mathfrak{g},B)$.

Then $\mathfrak{g}_\alpha=(\g,[~,~]_\alpha,\alpha,B_\alpha)$, where
for any $x,y\in \g$
\begin{eqnarray}[x,y]_\alpha=[\alpha(x),\alpha(y)]\\
B_\alpha(x,y)=B(\alpha(x),y), \end{eqnarray} is a quadratic Hom-Lie
algebra.
\end{proposition}
\begin{proof}The triple $(\g,[~,~]_\alpha,\alpha)$ is a Hom-Lie algebra
by Theorem \ref{thmYauConstrHomLie}.

The linear form $B_\alpha$ is nondegenerate since $B$ is
nondegenerate and $\alpha$ bijective.

We show that the identity (\ref{InvariantForm}) is satisfied by
$\mathfrak{g}_\alpha=(\g,[~,~]_\alpha,\alpha,B_\alpha)$. Let
$x,y,z\in \g$, then
\begin{eqnarray*}
B_\alpha([x,y]_\alpha,z)&=&B(\alpha([\alpha(x),\alpha(y)]),z)
\\ \
&=& B([\alpha(x),\alpha(y)],\alpha(z))
\\ \
&=& B(\alpha(x),[\alpha(y),\alpha(z)]) \quad (\text{Invariance of } B)
\\ \
&=& B(\alpha(x),[y,z]_\alpha)
\\ \
&=& B_\alpha(x,[y,z]_\alpha).
\end{eqnarray*}
Therefore $B_\alpha$ is invariant.

We have $\alpha \in Aut_S(\mathfrak{g}_\alpha,B_\alpha)$. Indeed
$$\alpha ([x,y]_\alpha)=\alpha ([\alpha (x),\alpha (y)])=
[\alpha^2 (x),\alpha^2 (y)]=[\alpha (x),\alpha (y)]_\alpha,$$ and
$$B_\alpha(\alpha(x),y)=B(\alpha(\alpha(x)),y)=B(\alpha(x),\alpha(y))=
B_\alpha(x,\alpha(y)).$$
\end{proof}

The following theorem permits to obtain new quadratic Hom-Lie algebras starting from a multiplicative quadratic  Hom-Lie algebra.
\begin{proposition}
Let  $\left( \g,[\ ,\ ] ,\alpha,B \right) $ be a multiplicative quadratic  Hom-Lie algebra. For any $n\geq 0$, the  quadruple
\begin{equation}\label{QuadraticDerivedHomAlgtype1}
\g _{(n)}=\left( \g,[\ ,\ ]^{(n)}=\alpha^{n }\circ[\ ,\ ] ,\alpha^{n+1}, B_{\alpha^n}\right),
\end{equation}
where $B_{\alpha^n}$ is defined for $x,y\in\g$ by $B_{\alpha^n}(x,y)=B(\alpha^n (x),y)$, determine a multiplicative quadratic Hom-Lie algebra.
\end{proposition}
\begin{proof}
The triple $\g _{(n)}=\left( \g,[\ ,\ ]^{(n)}=\alpha^{n }\circ[\ ,\ ] ,\alpha^{n+1} \right)$ is a Hom-Lie algebra by Theorem \ref{ThmConstrNthDerivedLie}.

Since $\alpha \in Aut(\g )$ by induction we have $\alpha^n \in Aut(\g )$. The bilinear form $B_{\alpha^n}$ is nondegenerate because $B$ is nondegenerate and $\alpha^n$ is bijective. It is  is symmetric. Indeed
$$B_{\alpha^n}(x,y)=B(\alpha^n (x),y)=B(x, \alpha^n (y))=B( \alpha^n (y),x)=B_{\alpha^n}(y,x).$$

The invariance of $B_{\alpha^n}$ is given by
\begin{align*}
B_{\alpha^n}([x,y]^n, z)&=B(\alpha^n \circ\alpha^n([x,y]), z)=B( \alpha^n([x,y]),\alpha^n( z))=B( [\alpha^n(x),\alpha^n(y)],\alpha^n( z))\\
\ &=B( \alpha^n(x),[\alpha^n(y),\alpha^n( z)])=B( \alpha^n(x),\alpha^n([y,z]))=B_{\alpha^n}(x,[y, z]^n).
\end{align*}
We have also $B_{\alpha^n}(\alpha ^n (x),y)=B_{\alpha^n}(x,\alpha ^n (y))$, indeed
 $$B_{\alpha^n}(\alpha ^n (x),y)=B(\alpha ^{2n} (x),y)=B(\alpha ^{n} (x),\alpha ^{n} (y))=B_{\alpha^n}(x,\alpha ^n (y)).$$
\end{proof}
We provide here a construction a   Hom-Lie algebra $\mathcal{L}$ which is a generalization of the  trivial $T^*$-extension   introduced by M. Bordemann \cite{Bordemann97} and also the double extension of $\{0\}$ by $\mathcal{L}$ see \cite{MedinaRevoy}.
\begin{proposition}
Let $(\g,[\ ,\ ]_\mathfrak{g} )$ be a Lie algebra and  $\g^*$ be the underlying  dual vector space. The vector space $\mathcal{L}=\g\oplus \g^*$ equipped with the following  product
\begin{equation}\
[ \ ,\ ]: \begin{array}{c}
 \mathcal{L} \times \mathcal{L} \rightarrow \mathcal{L}\\
(x+f,y+h)\mapsto [x,y]_{\mathfrak{g}}+f\circ ady-h \circ adx
\end{array}
\end{equation}
and a bilinear form
\begin{equation}
B: \begin{array}{c}
\mathcal{L}\times\mathcal{L}\rightarrow \K \\
(x+f,y+h)\mapsto f (y)+h(x)
\end{array}
\end{equation}
is a quadratic Lie algebra, which we denote by $\mathcal{L}.$
\end{proposition}

 In the sequel we denote $\mathcal{L}$ by $T^*(\mathfrak{g})$ and $B$ by $B_0$.

\begin{theorem}\label{Thm1}
Let $(\g,[\ ,\ ])$ be a Lie algebra and $\alpha \in Aut(\mathfrak{g}).$ Then the endomorphism $\Omega := \alpha + \ ^t\alpha $ of $T^*(\mathfrak{g})$ is a symmetric automorphism of $T^*(\mathfrak{g})$ with respect to $B_0$ if and only if $Im(\alpha^2-id)\subseteq \mathcal{Z}(\mathfrak{g})$,  where $\mathcal{Z}(\mathfrak{g})$ is the center of $\mathfrak{g}$.

Hence, if  $Im(\alpha^2-id)\subseteq \mathcal{Z}(\mathfrak{g})$ then $( T_0^*(\mathfrak{g})_\Omega,[\ ,\ ]_\Omega,\Omega,B_\Omega )$ is a RQH-Lie algebra where $\Omega=\alpha+\ ^t\alpha.$
\end{theorem}

\begin{proof}
Let $x,y\in \g$ and $f,h\in \g^*$.

\begin{eqnarray*}
\Omega ([x+f,y+h])&=&\Omega([x,y]_{\mathfrak{g}}+f\circ ady-h \circ adx)\\
\ &=& \alpha([x,y]_{\mathfrak{g}})+f\circ ady\circ \alpha -h \circ adx \circ \alpha,
\end{eqnarray*}

and
\begin{eqnarray*}
[\Omega (x+f),\Omega (y+h)]&=&[ \alpha(x)+f\circ \alpha, \alpha(y)+h\circ \alpha]\\
\ &=& [\alpha (x),\alpha (y)]_{\mathfrak{g}}+f\circ \alpha \circ ad \alpha (y) -h \circ \alpha\circ ad \alpha (x) ,
\end{eqnarray*}
Then $\Omega ([x+f,y+h])=[\Omega (x+f),\Omega (y+h)]$ if and only if
$$\forall x,y\in \g,\quad f\circ ady\circ \alpha -h \circ adx \circ \alpha=f\circ \alpha \circ ad \alpha (y) -h \circ \alpha\circ ad \alpha (x).
$$
That is for all $z\in \g$
$$f([y,\alpha (z)])-h([x,\alpha (z)])=f(\alpha[\alpha(y),z])-h(\alpha[\alpha(x),z]).$$
Hence, $\Omega$ is an automorphism of $T^*(\mathfrak{g})$ if and only if $f([x,\alpha(y)])=f(\alpha[\alpha(x),y])$, $\forall f\in \g^*$ $\forall x,y \in \g,$
  which is equivalent to $[x,\alpha(y)]=\alpha[\alpha(x),y]$ $\forall x,y \in \g.$

As a consequence, $\Omega \in  Aut(T_0^*(\mathfrak{g}))$ if and only if $[\alpha^2(x)-x,\alpha (y)]_\g=0$ $\forall x,y \in \g$, ie. $Im(\alpha^2 -id)\subseteq \mathcal{Z}(\mathfrak{g})$, since $\alpha \in Aut(\mathfrak{g}).$

In the following we show that $\Omega$ is symmetric with respect to $B_0$. Indeed, let $x,y\in \g$ and $f,h\in \g^*$

\begin{eqnarray*}
B_0(\Omega (x+f),y+h)&=&B_0(\alpha(x)+f\circ \alpha,y+h)
\\ \
&=& f\circ \alpha(y)+h(\alpha(x))
\\ \
&=& f\circ \alpha(y)+h\circ \alpha(x)
\\ \
&=& B_0(x+f,\alpha(y)+h\circ \alpha)
\\ \
&=& B_0(x+f,\Omega (y+h)).
\end{eqnarray*}

The last assertion  is a consequence of the previous calculations and  Proposition \ref{prop2}.
\end{proof}

In the following we provide examples which show that the class of Lie algebras with automorphisms satisfying the condition $Im (\alpha^2(x)-x) \in \mathcal{Z}(\mathfrak{g})$ is large. We consider first Lie algebras with involutions.

\begin{corollary}\label{cor4.5}
Let $(\g,[\ ,\ ]_\mathfrak{g})$ be a Lie   algebra and $\theta\in Aut(\g )$ such that $\theta ^2=id$ ($\theta$ is an involution), then $\theta^2(x)-x= 0 \in \mathcal{Z}(\mathfrak{g})$, $\forall x\in \g$. Thus $( T_0^*(\mathfrak{g})_\Omega,[\ ,\ ]_\Omega,\Omega,B_\Omega )$ is a RQH-Lie algebra where $\Omega=\alpha+\ ^t\alpha.$
\end{corollary}

\begin{example}
Recall that considering an involution on a Lie algebra $\g$ is equivalent to have a $\Z _2$-graduation on $\g$. The Lie algebras with involutions are called symmetric (see \cite{Dixmier2} \cite{TauvelLivre}).

It  is well known that starting from a Lie algebra one may construct a symmetric Lie algebra in the following way :

Let $( \g, [\cdot,\cdot])$ be a Lie algebra, we consider the Lie algebra $( \LL, [\cdot,\cdot]_\LL)$ where $\LL=\g \times \g$ and the bracket defined by for all $x,y,x',y'\in \g$ by $[(x,y),(x',y')]_\LL:=([x,x'],[y,y']).$

It is easy to check that the map $\
\theta: \begin{array}{c}
 \mathfrak{L}  \rightarrow \mathfrak{L}\\
(x,y)\mapsto (y,x)
\end{array}
$ is an automorphism of $\LL$.

Then the trivial  $T^*$-extension of $\LL$ has $\Omega=\theta+\ ^t \theta$ as a symmetric automorphism with respect to $B_0$. Moreover, $\Omega$ is an involution. According to corollary  \ref{cor4.5}, we have  $( T_0^*(\LL)_\Omega,[\ ,\ ]_\Omega,\Omega,(B_0)_\Omega )$ is a quadratic Hom-Lie algebra.
\end{example}

\begin{example}\label{exampleSS1}
Let $( \g, [\cdot,\cdot])$ be a semisimple Lie algebra with an involution $\theta$. Recall that the classification of semisimple Lie algebras with  involutions could be found in \cite{Lister}.
The Killing form $\mathcal{K}$ of $\g$ is nondegenerate and $\theta$ is symmetric with respect to $\mathcal{K}$. Then $(\g_\theta,[\ ,\ ]_\theta,\theta,\mathcal{K}_\theta)$ is a RQH-Lie algebra.

For example, let $\g=\mathfrak{sl}_n(\K)$ (with $n\geq 2$). The linear map  $\theta : \g\rightarrow\g$ defined for all $x\in \g$ by $\theta (x)=-^t x$ is an involution automorphism. Therefore we endow  $(\mathfrak{sl}_n(\K))_\theta$ with a RQH-Lie algebra structure.
\end{example}

\begin{example}
We show an example of Lie algebra $\g$ with automorphisms $\alpha$ which are not involutions and satisfying $Im(\alpha^2-id)\subseteq \mathcal{Z}(\mathfrak{g})$.

Let $\g$ be a finite-dimensional vector space and $\B=\{x_0,\cdots, x_n\}$, $(n\geq 1)$ be a basis of $\g$. We define on $\g$ a structure of Lie algebra by $$[x_0,x_i]=x_{i+1},\quad \forall i\in \{1,\cdots, n-1\}$$
The others brackets are defined obtained by skewsymmetry or equal to zero.

This Lie algebra is nilpotent and called filiform Lie algebra, its center is the subvector space generated by $<x_n>$.

The endomorphism $\alpha : \g \rightarrow \g $ defined by
\begin{align*}
& \alpha (x_0)=x_0 +\lambda x_n,\quad  \text{where } \lambda\in \K\setminus \{0\},\\
& \alpha (x_i)=x_i, \quad \forall i\in \{1,\cdots, n\}
\end{align*}
is an automorphism of $\g$ satisfying
\begin{align*}
& \alpha^2 (x_0)=x_0 +2 \lambda x_n,\\
& \alpha^2 (x_i)=x_i, \quad \forall i\in \{1,\cdots, n\}
\end{align*}
Therefore $(\alpha^2-id)(\g)\subseteq \mathcal{Z}(\mathfrak{g})$ and $\alpha^2 \neq id$. According to Theorem \ref{Thm1},   $( T_0^*(\LL)_\Omega,[\ ,\ ]_\Omega,\Omega,(B_0)_\Omega )$, where $\Omega=\alpha+^t \alpha$, is a RQH-Lie algebra.
\end{example}

\begin{example}[Nonabelian $2$-nilpotent Lie algebras]
 Let $\g=V\oplus \mathcal{Z}(\g)$, where $V\neq\{0\}$ is a subspace  of the vector space $\g$ with $[V,V]=[\g,\g]\subseteq \mathcal{Z}(\g)$.

Let $\lambda : \g\rightarrow \mathcal{Z}(\g)$ be a nontrivial  linear map and $\alpha : \g\rightarrow \g$ is an endomorphism of $\g$ defined by
$$\alpha (v+z):=v+\lambda (v)+z\quad \forall v\in V\ \  \forall z\in \mathcal{Z}(\g).
$$
We have $\alpha ([v+z,v'+z'])=\alpha ([v,v'])=[v,v']$ since $[v,v']\in \mathcal{Z}(\g)$.
Also $[\alpha (v+z),\alpha (v'+z')])=[v,v']$. Therefore, the map $\alpha$ is an injective Lie algebra morphism. Thus $\alpha$ is an automorphism of $\g$.

Moreover, if $v\in \g$ and $z\in \mathcal{Z}(\g)$, we  have
\begin{align*}
(\alpha^2-id)(v+z)&= \alpha^2(v+z)-(v+z)\\
\ &= \alpha(v+\lambda (v)+z)-(v+z)\\
\ &= v+2 \lambda (v)+z-v-z)\\
\ &= 2 \lambda (v).
\end{align*}
Then $\alpha^2-id\neq 0$ and $Im(\alpha^2-id)\subseteq \mathcal{Z}(\g).$
It follows that $( T_0^*(\g)_\Omega,[\ ,\ ]_\Omega,\Omega,(B_0)_\Omega )$, where $\Omega=\alpha+^t \alpha$, is a RQH-Lie algebra.

It is clear that $T_0^*(\g)_\Omega$ is 2-nilpotente. It 's also a quadratic Lie algebra.
\end{example}

%
\begin{proposition}
Let $\A$ be a commutative associative algebra and $\g$ be a Lie algebra.

If $\A$ has an automorphism $\theta$ such that $Im(\theta^2-id)\subseteq Ann (\A)$, where $Ann(\A)$ denotes the annihilator of $\A$, then the endomorphism $\widetilde{\theta} :=id_\g\otimes \theta$ of $\g\otimes \A$ is an automorphism of the Lie algebra $( \g\otimes \A,[\ ,\ ])$, where $[x\otimes a,y\otimes b]:= [x,y]_\g \otimes ab$ for all $x,y\in \g$ and  $a,b\in\A$. In addition, $Im(\widetilde{\theta}^2-id_{\g \otimes\A})\subseteq \mathcal{Z}(\g\otimes\A)$. Then
$( T_0^*(\g\otimes \A)_\Omega,[\ ,\ ]_\Omega,\Omega,(B_0)_\Omega )$ is a RQH-Lie algebra.

Moreover, if $\theta^2\neq id_\A$ then $\widetilde{\theta}^2\neq id_{\g \otimes\A}$.
\end{proposition}
\begin{proof}
It follows from direct calculation and Theorem \ref{Thm1}.
\end{proof}

\begin{example}
Let $\A$ be the vector space spanned by $\{e,f,h,t\}$ on which we define a commutative associative multiplication by
$$ee=f,\quad ef=fe=h,\quad  eh=he=t,\quad ff=t.$$
We consider the  algebra morphism $\theta :\A \rightarrow \A$ defined by
 $$\theta (e)=e+qt,\quad \theta(f)=f,\quad  \theta (h)=h,\quad \theta (t)=t.$$
 where $q\in \K \setminus \{0\}.$
 It is easy to check that $Im(\theta^2-id)\subseteq Ann (\A)$ and $\theta^2\neq id.$

 Let $\g$ be a Lie algebra (for example $\g = \mathfrak{sl}(2)$), then  $\g\otimes \A$ is a Lie algebra with $\widetilde{\theta}=id_\g \otimes \theta \in Aut(\g\otimes \A)$ satisfying $Im(\widetilde{\theta}^2-id_{\g \otimes\A})\subseteq \mathcal{Z}(\g\otimes\A)$. Then
$( T_0^*(\g\otimes \A)_\Omega,[\ ,\ ]_\Omega,\Omega,(B_0)_\Omega )$ is a RQH-Lie algebra.
\end{example}


\section{Connection between quadratic Lie algebras and  Hom-Lie algebras}
We establish a connection between some classes of Lie algebras (resp. quadratic Lie algebras) and  classes of  Hom-Lie algebras (resp. quadratic Hom-Lie algebras).
\subsection{Lie algebras with involutive automorphisms}
\begin{theorem}
There exists a biunivoque correspondence between the class of Lie algebras (resp. quadratic Lie algebras) admitting involutive automorphisms (resp. symmetric involutive automorphisms) and the class of  Hom-Lie algebras (resp. quadratic Hom-Lie algebras) where twist maps are involutive automorphisms (resp. symmetric involutive automorphisms)

\end{theorem}

\begin{proof}
Let $(\g,[ \ ,\ ])$ be a symmetric Lie algebra with  $\theta$  an involutive automorphism of $\g$.

Then, according to Theorem \ref{thmYauConstrHomLie}, $(g_\theta,[\ ,\ ]_\theta,\theta)$ is a Hom-Lie algebra where $\theta$ is an involutive automorphism of $\g_\theta$.

Moreover, if $\g$ has an invariant scalar product $B$ such that $\theta$ is symmetric with respect to $B$, we have seen that
\begin{equation}
B_\theta: \begin{array}{c}
\g_\theta\times\g_\theta\rightarrow \K \\
(x,y)\mapsto B_\theta(x,y):=B(\theta(x),y)
\end{array}
\end{equation}
defines a quadratic structure on $\g_\theta$.

Conversely, let $(H,[\ ,\ ]_H, \theta)$ be a Hom-Lie algebra where $\theta$ is an involutive automorphism of $H$.

We will untwist the Hom-Lie algebra structure by considering the vector space $H$ and the bracket
\begin{equation}
[\ ,\ ]: \begin{array}{c}
H\times H \rightarrow H \\
(x,y)\mapsto [x,y]:=[\theta(x),\theta (y)]_H
\end{array}
\end{equation}
Obviously the new bracket is bilinear and skewsymmetric. We show that it satisfies the Jacobi identity.

Indeed, for $x,y,z\in H$ we have
\begin{align*}
[x,[y,z]]&= [\theta (x),\theta ([y,z])]_H\\
\ &= [\theta (x),\theta ([\theta(y),\theta(z)]_H)]_H\\
\ &= [\theta (x),[\theta^2(y),\theta^2(z)]_H)]_H\\
\ &= [\theta (x),[y,z]_H)]_H.
\end{align*}
Thus
$$\circlearrowleft_{x,y,z}{[x,[y,z]]}=\circlearrowleft_{x,y,z}{[\theta (x),[y,z]_H)]_H}=0.$$
Thus $(H,[\ ,\ ])$ is a Lie algebra.

Furthermore, for $x,y\in H$
$$\theta ([x,y])=\theta ([\theta(x),\theta(y)]_H)=[\theta^2(x),\theta^2(y)]_H =[x,y]_H$$
and
$$[\theta(x),\theta(y)]=[\theta^2(x),\theta^2(y)]_H=[x,y]_H.$$
Then $\theta ([x,y])=[\theta(x),\theta(y)]$. Therefore $\theta$ is an involutive automorphism of the Lie algebra  $(H,[\ ,\ ])$.

Also for $x,y\in H$
\begin{align*} [x,y]_\theta:&=[\theta(x),\theta(y)]
=[\theta^2(x),\theta^2(y)]_H
=[x,y]_H.
\end{align*}
Then $(H,[\ ,\ ]_\theta,\theta)$ is the Hom-Lie algebra  $(H,[\ ,\ ]_H,\theta)$.

Now, let $(H,[\ ,\ ]_H, \theta,B)$ be a quadratic  Hom-Lie algebra.

The bilinear form
\begin{equation}
T: \begin{array}{c}
H\times H \rightarrow \K \\
(x,y)\mapsto T(x,y)=B(\theta(x),y)
\end{array}
\end{equation}
is symmetric and nondegenerate.

Indeed, for Let
$x,y,z\in H$, we have
\begin{eqnarray*}
T([x,y],z)&=&B(\theta([x,y]),z)
\\ \
&=& B(\theta[\theta(x),\theta(y)]_H,z)
\\ \
&=& B([x,y]_H,z)
\\ \
&=& B(x,[y,z]_H)
\\ \
&=& B(\theta (x),\theta ([y,z]_H)) \quad \text{ $\theta$ is  $B$-symmetric}
\\ \
&=& B(\theta (x),[\theta (y),\theta (z)]_H))
\\ \
&=& B(\theta (x),[y,z]))
\\ \
&=& T(x,[y,z]).
\end{eqnarray*}
Then $T$ is invariant.   In the other hand,
$$T(\theta (x),y)=B(x,y)=B(\theta (x), \theta (y))=T(x,\theta (y)).$$
That is $\theta$ is symmetric with respect to $T$.

Therefore $(H,[\ ,\ ], T)$ is a quadratic Lie algebra and  $(H,[\ ,\ ]_\theta, \theta,T_\theta)$ is an IQH-Lie algebra.
\end{proof}

\subsection{Quadratic Hom-Lie algebras and Centroid's elements}

Now we discuss the connection between Hom-Lie algebras where the twist map is in the centroid and quadratic Lie algebras.
 Let $(\g,[\ ,\ ],B)$ be a quadratic Lie algebra and $\theta\in Cent(\g )$ such that $\theta$ is invertible and symmetric with respect to $B$. We set
 $$ Cent_S(\g)=\{ \theta\in Cent(\g) : \theta \text{ symmetric with respect to } B\}.
 $$

 We consider
 \begin{equation}
B_\theta: \begin{array}{c}
\g\times\g\rightarrow \K \\
(x,y)\mapsto B_\theta(x,y):=B(\theta(x),y)
\end{array}
\end{equation}
We have
\begin{itemize}
\item $B_\theta$ is symmetric.
\item $B_\theta$ is nondegenerate.
\item $B_\theta$ is invariant, indeed
\begin{align*}
B_\theta(\{x,y\},z) &=B_\theta([\theta(x),y],z)=B(\theta([\theta (x),y]),z)\\
\ &=B([\theta (x),y],\theta (z))=B(\theta (x),[y,\theta (z)])\\
\ & = B(\theta(x),[\theta (y),z])=B(\theta(x),\{y,z\})\\
\ &= B_\theta(x,\{y,z\}).
\end{align*}
\end{itemize}
Also, we have
$$B_\theta(\theta(x),y)=B(\theta^2(x),y)=B(\theta(x),\theta (y))=B_\theta(x,\theta(y)).
$$
Then $(\g,\{\ ,\ \}, \theta,B_\theta )$ is a quadratic Hom-Lie algebra.

Notice that $B_\theta $ is also an invariant scalar product of the Lie algebra $\g$.

We have also that $(\g,[\  ,\ ]_\theta, \theta,B_\theta )$ is a quadratic Hom-Lie algebra. indeed
\begin{align*}
B_\theta([x,y]_\theta,z) &=B_\theta([\theta(x),\theta (y)],z)=B(\theta([\theta (x),\theta (y)]),z)\\
\ &=B([\theta (x),\theta (y)],\theta (z))=B(\theta (x),[\theta (y),\theta (z)])\\
\ & = B(\theta(x),[y,z]_\theta)\\
\ &= B_\theta(x,[y,z]_\theta).
\end{align*}
Observe that
\begin{align*}
\theta ([x,y]_\theta)=\theta [\theta (x),\theta (y)]=[\theta^2 (x),\theta (y)]=[\theta (x),y]_\theta.\\
\theta (\{x,y\})=\theta [\theta (x),y]=[\theta^2 (x),y]=\{\theta (x),y\}.
\end{align*}
We may say that $\theta\in Cent(\g,\{\ ,\ \} )$ and $\theta\in Cent(\g,[\ ,\ ]_\theta ).$

Conversely, let $(\g,[\ ,\ ],\alpha )$ be a Hom-Lie algebra such that $\alpha \in Cent(\g,[\ ,\ ],\alpha )$.

We define a new bracket as
$\{x,y\}:=[\alpha (x),y].
$. Then $(\g, \{\ ,\ \}))$ is a Lie algebra. Indeed
\begin{itemize}
\item The bracket is skewsymmetric.
\item We have
\begin{align*}
\{x,\{y,z\}\}&=[\alpha (x),[\alpha(y),z]],\\
\{y,\{z,x\}\}&=[\alpha (y),[\alpha(z),x]=[\alpha^2(y),[z,x]],\\
\{z,\{x,y\}\}&=[\alpha (z),[\alpha(x),y]]=[\alpha (z),[x,\alpha(y)]].
\end{align*}
Then
$$\circlearrowleft_{x,y,z}{\{x,\{y,z\}\}}=[\alpha (x),[\alpha(y),z]]+[\alpha^2(y),[z,x]]+[\alpha (z),[x,\alpha(y)]]=0.
$$
\end{itemize}

We may  define another  which gives rise to also a Lie algebra by
$[x,y]_\alpha:=[\alpha (x),\alpha(y)].
$.  Indeed the  bracket is skewsymmetric and we have
\begin{align*}
[x,[y,z]_\alpha]_\alpha&=[\alpha (x),\alpha([\alpha(y),\alpha (z)])]=[\alpha (x),[\alpha^2(y),\alpha (z)]]=[\alpha^2 (x),[\alpha(y),\alpha (z)]],\\
[y,[z,x]_\alpha]_\alpha&=[\alpha (y),\alpha([\alpha(z),\alpha (x)])]=[\alpha (y),[\alpha^2(z),\alpha (x)]]=[\alpha^2 (y),[\alpha(z),\alpha (x)]],\\
[z,[x,y]_\alpha]_\alpha&=[\alpha (z),\alpha([\alpha(x),\alpha (y)])]=[\alpha (z),[\alpha^2(x),\alpha (y)]]=[\alpha^2 (z),[\alpha(x),\alpha (y)]].
\end{align*}
Therefore
$$[\alpha^2 (x),[\alpha(y),\alpha(z)]]+[\alpha^2(y),[\alpha(z),\alpha(x)]]+[\alpha ^2(z),[\alpha(x),\alpha(y)]]=0.
$$
Now if there is an invariant scalar product $B$ on $(\g,[\ ,\ ] )$ and assume that $\alpha$ is invertible and symmetric with respect to $B$. Consider the bilinear form $B_\alpha$ defined by $B_\alpha (x,y)=B(\alpha (x),y)$. We have
\begin{eqnarray*}
B_\alpha(\{x,y\},z)&=&B(\alpha(\{x,y\})]),z)
\\ \
&=& B(\alpha([\alpha(x),y],z)
\\ \
&=& B(\alpha(x),[y,\alpha(z)])
\\ \
&=& B(\alpha(x),[\alpha (y),z])
\\ \
&=& B_\alpha(x,\{y,z\}).
\end{eqnarray*}
Similarly we have
\begin{eqnarray*}
B_\alpha([x,y]_\alpha,z)&=&B(\alpha([\alpha(x),\alpha(y)]),z)
\\ \
&=& B([\alpha(x),\alpha(y)],\alpha(z))
\\ \
&=& B(\alpha(x),[\alpha(y),\alpha(z)])
\\ \
&=& B(\alpha(x),[y,z]_\alpha)
\\ \
&=& B_\alpha(x,[y,z]_\alpha).
\end{eqnarray*}
Therefore $(\g,\{\ ,\ \} ,B_\alpha)$ and $(\g,[\ ,\ ]_\alpha ,B_\alpha)$ are quadratic Lie algebras.
Hence, we have the following theorem:
\begin{theorem}
There exists a biunivoque correspondence between the class of Lie algebras (resp. quadratic Lie algebras) admitting an  element in the centroid (resp. symmetric  invertible element in the centroid) and the class of  Hom-Lie algebras (resp. quadratic Hom-Lie algebras) where twist map is in the centroid (resp. symmetric  invertible element in the centroid).
\end{theorem}
An interesting case is when the element $\theta$  of the centroid is not of the form $\theta=k \ id _\g$ where $k\in\K$.  

One may replace the class of quadratic Lie algebras admitting a  symmetric  invertible element $\theta$ in the centroid,  such that $\theta\neq k \ id_\g$ where $k\in\K $, by the class of quadratic Lie algebras of  quadratic dimension larger than 2, i.e. there exist $B$ and $B'$ two invariant scalar products on $\g$ such that $\nexists \lambda \in \K$ such that $B=\lambda B'.$  For quadratic dimension see \cite{BajoBenayadi,BenayadiJA2003,Pinczon10}.
\begin{corollary}
There exists a one to one  correspondence between the class of  quadratic Lie algebra of  quadratic dimension larger than 2 and the class of  quadratic Hom-Lie algebras where the twist map is invertible and  in the centroid  such that it is not a multiple of $id_\g$.
\end{corollary}
Consider a quadratic Lie algebras of  quadratic dimension larger than 2,  then  there exists $B$ and $B'$ two invariant scalar products on $\g$ such that it doesn't exist $ \lambda \in \K$ such that $B=\lambda B'.$ Since $B$ is nondegenerate then there exist $\theta\in End(\g)$ such that $B'(x,y)=B( \theta (x),y)$ and $\theta\neq k \ id _\g$ where $k\in\K$.

Observe that the invariance of$B$ and  $B'$ induces that $\theta$ is in the centroid of $\g.$ Indeed for any $x,y,z\in \g$
the identity $
B'([x,y],z)=B'(x,[y,z]) $
is equivalent  to
$B(\theta ([x,y]),z)=B(\theta (x),[y,z])=B([\theta (x),y],z).$
 It leads to $\theta ([x,y])=[\theta(x),y]$ since $B$ is non degenerate.
 Hence $B$ and $B'$ are nondegenerate implies that $\theta$ is invertible.

 Notice that $B'(x,y)=B'(y,x)$ implies $B(\theta (x),y)=B(\theta(y),x)=B(x,\theta (y)).$
 Conversely when we start with a Hom-Lie algebra where the twist map is invertible and in the centroid, we have constructed in the proof of the previous theorem two quadratic structures on this algebra.

\section{Simple and Semisimple Hom-Lie algebras}
In this section we give some observations about simple and semisimple Hom-Lie algebras.  In particular, we study simple and semisimple Hom-Lie algebras with involution. Simple involutive Hom-Lie algebras will appear in the last section when we will discuss a structure theorem of IQH-Lie algebra.

 A Hom-Lie algebra is said to be simple if it has no non-trivial ideals and is not abelian. A Hom-Lie algebra $\mathfrak{g}$  is called semisimple if its radical is zero (a radical is the  maximal solvable ideal). Equivalently, $\mathfrak{g}$ is semisimple if it does not contain any non-zero abelian ideals. In particular, a simple Hom-Lie algebra is semisimple.
In \cite{JinLi}, the authors studied  when a finite-dimensional semi-simple Lie algebra admits non-trivial hom-Lie algebra structures.

\begin{proposition}\label{HomSimple1}
Let $( \g, [\cdot,\cdot])$ be a Lie algebra and $\alpha\in Aut(\g)$. If $\g$ is simple then  $\g_\alpha$ obtained by composition method is simple.
\end{proposition}
\begin{proof}
Indeed, let $I$ be an ideal in $\g_\alpha$, that is $[\g_\alpha,I]\subseteq I$ and $\alpha (I)\subseteq I$. Then, $\forall x\in \g$ and  $\forall a\in I$ we have $[\alpha(x),\alpha(a)]\in I$. Therefore $I$ is an ideal of $\g$ because $\alpha(I)=I  $. Hence, $I=\{0\}$ or $I=\g=\g_\alpha.$

Since $[\g,\g]\neq \{0\}$, then $[\g_\alpha,\g_\alpha]_\alpha\neq \{0\}$. Thus, $\g_\alpha$ is a simple Hom-Lie algebra.
\end{proof}
Hence the previous proposition  provides a way to construct simple Hom-Lie algebras, see example \ref{exampleSS1}.

We have proved in Proposition \ref{HomSimple1} that a simple Lie algebra with an automorphism (in particular an involution)  gives rise to a simple Hom-Lie algebra using composition method. The converse is discussed further  according to Proposition
\ref{CompoMethodLie}.

Let $(\g,[\ ,\ ] ,\theta)$ be a simple  involutive Hom-Lie algebra. We denote by $\rho (x)$, for $x\in\g $, the linear map
$\rho (x) :\g \rightarrow \g$ defined for $y\in\g$ by $\rho (x)(y)=[x,y].$

We also set $B$ to be  the map $B:\g \times \g \rightarrow \K$ defined for $x,y\in\g$ by $B(x,y)=tr (\rho (x)\rho (y)).$ Obviously $B$ is bilinear and symmetric. We show now that it is also invariant. Let $x,y,z\in\g$ then
\begin{align*}
\rho ([x,y])(z)&=[[x,y],z]=-[z,[x,y]]\\
\ &=-[\theta(\theta(z)),[x,y]]=[\theta(x),[y,\theta(z)]]+[\theta(y),[\theta(z),x]]\\
\ &=\theta([x,[\theta(y),z]])+\theta([y,[z,\theta(x)]])\\
\ &=(\theta \circ \rho (x)\circ \rho (\theta(y))-\theta \circ \rho (y)\circ \rho (\theta(x)))(z).
\end{align*}
In addition we have $\rho(\theta(x))(z)=[\theta(x),z]=\theta([x,\theta(z)])=\theta\circ\rho(x)\circ\theta(z)$ which leads for any $x\in\g$ to
$$ \rho(\theta(x))=\theta\circ\rho(x)\circ\theta.$$
Since for any $x,y\in\g$ we have
$$\rho ([x,y])=\theta \circ \rho (x)\circ \rho (\theta(y))-\theta \circ \rho (y)\circ \rho (\theta(x)),$$
it follows  for any $x,y,z\in\g$
\begin{eqnarray*}
\rho ([x,y])\rho(z)&&=(\theta \circ \rho (x)\circ \rho (\theta(y))-\theta \circ \rho (y)\circ \rho (\theta(x)))\rho(z),\\
\ &&=\theta \circ \rho (x)\circ\theta\circ\rho(y)\circ\theta\circ\rho(z)-\theta \circ \rho (y)\circ\theta\circ\rho(x)\circ\theta\circ\rho(z).
\end{eqnarray*}
Therefore
\begin{eqnarray*}
tr (\rho ([x,y])\rho(z))&&=tr (\theta \circ \rho (x)\circ\theta\circ\rho(y)\circ\theta\circ\rho(z))-tr (\theta \circ \rho (y)\circ\theta\circ\rho(x)\circ\theta\circ\rho(z)),\\
\ &&=tr (\theta\circ\rho(y)\circ\theta\circ\rho(z)\circ\theta \circ \rho (x))-tr (\theta\circ\rho(z)\circ\theta \circ \rho (y)\circ\theta\circ\rho(x)),\\
\ &&=tr( (\theta\circ\rho(y)\circ\theta\circ\rho(z)\circ\theta-\theta\circ\rho(z)\circ\theta \circ \rho (y)\circ\theta)\circ\rho(x)),\\
\ &&=tr (\rho([y,z])\circ\rho(x)),\\
\ &&=tr (\rho(x)\circ\rho([y,z])).
\end{eqnarray*}
Which proves that $B([x,y],z)=B(x,[y,z])$, that is $B$ is invariant.

Let us consider $B_\theta:\g \times\g \rightarrow \K$ defined for any $x,y\in\g$ by $B_\theta (x,y)=B(\theta (x),y).$

We have for $x,y,z\in\g$
\begin{eqnarray*}
\rho (\theta(x))\rho(y)(z)&&=[\theta(x),[y,z]]=[x,\theta([y,z])]_\theta,\\
\ &&=[\theta^2(x),[\theta(y),\theta(z)]]_\theta,\\
\ &&=[\theta^2(x),[\theta^2(y),\theta^2(z)]_\theta]_\theta,\\
\ &&=[x,[y,z]_\theta]_\theta.
\end{eqnarray*}
Hence $\rho (\theta(x))\rho(y)=ad_{\g_\theta}(x)\circ ad_{\g_\theta}(y),$ where $g_\theta$ is the Lie algebra associated to $(\g,[\ ,\ ],\theta).$

We conclude that $B_\theta$ is the Killing form of the Lie algebra $g_\theta$.

Then $B_\theta=\mathcal{K}$ and  for $x,y\in\g,$  $B(x,y)=\mathcal{K}(\theta(x),y)=\mathcal{K}_\theta(x,y).$

\begin{theorem}\label{SimpleInvolutiveHomLie}
Let $(\g,[\ ,\ ] ,\theta)$ be a simple  involutive Hom-Lie algebra then $g_\theta=(\g,[\ ,\ ]_\theta )$ is either a simple Lie algebra or a semisimple Lie algebra.
Moreover in the second case it decomposes into $S\oplus \theta (S)$ where $S$ is a simple ideal.

In addition  the form  $B:\g \times \g \rightarrow \K$, defined for $x,y\in\g$ by $B(x,y)=tr (\rho (x)\rho (y))$  where $\rho(x)=[x,\cdot ]$, defines a quadratic structure on  $(\g,[\ ,\ ] ,\theta)$ and $B(x,y)=\mathcal{K}(\theta (x),y)$, where $\mathcal{K}$ is the Killing form of the Lie algebra $g_\theta=(\g,[\ ,\ ]_\theta)$.

Conversely, if $(\g, [\, \ ])$ is a simple Lie algebra and $\theta\in Aut(\g)$  is an involution then $(\g,[\ ,\ ], \theta)$ is a simple Hom-Lie algebra.
\end{theorem}
\begin{proof}
Assume that $g_\theta$ is not simple, then $g_\theta$ contains a minimal ideal $S$. Notice that $g_\theta$ is equal to $\g$ as a linear space, we refer by $\g$ to the Hom-Lie algebra and by $\g_\theta$ for the corresponding Lie algebra by composition.  Recall that $I$ is a minimal ideal of $g_\theta$ if $I\neq\{0\}$, $I\neq\g_\theta$ and if $J$ is another ideal of $g_\theta$ such that $J\subseteq I$ then $J=\{0\}$ or $J=I.$

We do have $[g_\theta,S]_\theta$ is an ideal of $g_\theta$  such that $[g_\theta,S]_\theta\subseteq S$. Since $S$ is minimal then $[g_\theta,S]_\theta=\{0\}$ or $[g_\theta,S]_\theta= S.$

If $[g_\theta,S]_\theta=\{0\}$ then $[\theta(g_\theta),\theta(S)]=\{0\}$ which implies $[\g,\theta(S)]=\{0\}$ and $\theta(S)\subseteq\mathcal{Z}(g).$ Hence $\mathcal{Z}(g)\neq\{0\}$ since $\theta(S)\neq\{0\}$. This contradicts the fact that $(\g,[\ ,\ ],\theta)$ is a simple Hom-Lie algebra.

Consequently we have  $[g_\theta,S]_\theta= S$. It follows $[\g,\theta(S)]=S$, then  $\theta([\g,\theta(S)])=\theta(S).$ Thus $[\theta(\g),\theta^2(S)])=[\g,S]=\theta(S).$ In addition, $\theta(S+\theta(S))=\theta(S)+\theta^2(S)=\theta(S)+S.$

Then $S+\theta(S)$ is an ideal of  $(\g,[\ ,\ ] ,\theta)$ and $S+\theta(S)\neq\{0\}$. Therefore $\g=S+\theta(S).$

We show that we have a direct sum. Since $\theta$ is an automorphism of $\g_\theta$ then $\theta(S)$ is an ideal of $\g_\theta$.Thus $S\bigcap \theta (S)=\{0\}$ or $S\bigcap \theta (S)=S$ since $S$ is a minimal ideal in $g_\theta.$
Assume $S\bigcap \theta (S)=S$, then $S\subseteq \theta (S)$ which leads to   $S= \theta (S)$ since $\theta$ is bijective.

We do have
$[\g,S]=\theta([\theta(\g),\theta(S)]=\theta([\g, S]_\theta)\subseteq\theta(S)=S$
 and $\theta(S)=\theta^2(S)=S.$ Then $S$ is an ideal of $(\g,[\ ,\ ] ,\theta)$. Thus $S=\g$ which contradicts the assumption $S\neq\g$.  Hence $S\bigcap \theta (S)=\{0\}$ and $\g=S\oplus\theta(S).$

 Consequently, we do have $\g_\theta=S\oplus\theta(S).$ Indeed $\theta$ is an automorphism of $\g$ then $\theta$ is an automorphism of $\g_\theta$. Therefore $\theta(S)$ is an ideal of $\g$. We do have $[\g_\theta,S]_\theta=S$ then $[S\oplus\theta(S),S]_\theta=S$ which implies $[S,S]_\theta=S$ since $[\theta(S),S]_\theta=\{0\}.$

 Therefore $S$ is a simple ideal because $S$ is a minimal ideal such that $[S,S]_\theta=S.$

 It follows that $\g_\theta$ is a semisimple Lie algebra because $S$ and $\theta(S)$ are simple ideals of $\g_\theta.$ We may view $\g_\theta$ as a $\mathbb{Z}_2$-graded simple Lie algebra.

 We have also that $\mathcal{K} $ is nondegenerate and consequently $B$ is non degenerate since $B(x,y)=\mathcal{K}(\theta(x),y)$ for $x,y\in\g$.
\end{proof}

We have proved that if $(\g , [\ ,\ ] )$ is a simple Lie algebra with involution  $\theta$ then $(\g_\theta , [\ ,\ ]_\theta ,\theta )$ is a simple Hom-Lie algebra. We have also proved that if $(\g , [\ ,\ ] ,\theta )$ is a simple Hom-Lie algebra then the Lie algebra $(\g_\theta , [\ ,\ ]_\theta )$ is either simple or $\g_\theta=S\oplus\theta (S)$, where $S$ is a simple ideal of the Lie algebra  $(\g_\theta , [\ ,\ ]_\theta )$, (in particular  $(\g_\theta , [\ ,\ ]_\theta )$ is a semisimple Lie algebra.

\begin{proposition}\label{Prop6.3} 
 Let $(\g , [\ ,\ ] )$ is a semisimple Lie algebra different from $\{0\}$ with an involution $\theta$ such that   $\g=S+\theta (S)$, where $S$ is a simple ideal of   $(\g , [\ ,\ ] )$. Then the Hom-Lie algebra $(\g_\theta , [\ ,\ ]_\theta ,\theta )$ is simple.
\end{proposition}
\begin{proof}
Let $I$ be an ideal of $\g_\theta$ such that $I\neq \{0\}$. Since $\theta (\g )=\g$ and $\g =\g_\theta$  then  $[\g ,\theta (I )]\subseteq I$ which implies that $[\g ,I]\subseteq I$. It follows that $I$ is an ideal of $\g$. Consequently  $I=S$ or $I=\theta (S)$ or $I=\g$. Since $\theta (I)=I$  then $I\neq S$ and $I\neq \theta (S)$, therefore $I=S+\theta (S)=\g$. Moreover $\g$ is semisimple then $[\g ,\g]=\g$ which implies $[\g_\theta ,\g_\theta]_\theta=\g_\theta$.Then $(\g_\theta , [\ ,\ ]_\theta ,\theta )$ is a simple Hom-Lie algebra.
\end{proof}

More generally,  let $(\g , [\ ,\ ] ,\theta )$ be an involutive  Hom-Lie algebra. We consider the associated Lie algebra $\g_\theta=(\g,[\ ,\ ]_\theta)$ where $[x,y]_\theta=[\theta(x),\theta(y)]$, $\forall x,y\in\g$. Let $\mathfrak{R}(\g)$ be the solvable radical of  $\g_\theta$. By Taft's Theorem (see \cite{taft}), there exists a Levi's component $\mathfrak{s}$ of  $\g_\theta$ invariant by $\theta$. It is clear that $\theta(\mathfrak{R}(\g))=\mathfrak{R}(\g)$.

The component $\mathfrak{s}$  is a semisimple Lie algebra then it may be written  $\mathfrak{s}=\mathfrak{s}_1\oplus\cdots \oplus \mathfrak{s}_n$ where $\{ \mathfrak{s}\}_{1\leq i\leq n}$ is the set of the simple ideals of $\mathfrak{s}$. The fact $\theta (\mathfrak{s})\subseteq \mathfrak{s}$ implies that $\theta_{| \mathfrak{s}}$ is an involutive automorphism of $\mathfrak{s}.$ Then for $i\in\{1,\cdots , n\}$ we have $\theta (\mathfrak{s}_i)$ is a simple ideal of the Lie algebra $\mathfrak{s}.$ Therefore for any $i\in\{1,\cdots , n\}$ there exist a unique $j_i\in\{1,\cdots , n\}$ such that $\theta (\mathfrak{s}_i)=\mathfrak{s}_{j_i}$. If $i=j_i$ then $\theta (\mathfrak{s}_i)=\mathfrak{s}_{i}$ and in  this case $(\mathfrak{s}_i,{ [\ ,\ ]_\theta}_{|\mathfrak{s}_i\times \mathfrak{s}_i})$ is a simple Lie algebra with $\theta_i=\theta_{|\mathfrak{s}_i}$ is an involutive automorphism of $\mathfrak{s}_i$. We can then assume that there exist $m,p\in \mathbb{N}$ such that $m+2p= n$ and
$$\mathfrak{s}=\mathfrak{s}_1\oplus\cdots\oplus\mathfrak{s}_m\oplus (\mathfrak{s}_{m+1}\oplus\theta(\mathfrak{s}_{m+1}))\oplus\cdots\oplus (\mathfrak{s}_{m+p}\oplus\theta(\mathfrak{s}_{m+p}))$$
where,  for $i\in \{1,\cdots ,m+p\}$, $\mathfrak{s}_i$ are simple ideals of $\mathfrak{s}$ and for $i\in \{1,\cdots ,m\}$ we have $\theta (\mathfrak{s}_i)=\mathfrak{s}_i.$

Let us set, for $i\in \{1,\cdots ,m\}$,  $ \mathfrak{S}_i=\mathfrak{s}_i$ and for $i\in \{1,\cdots ,p\}$,  $ \mathfrak{S}_{m+i}=\mathfrak{s}_{m+i}\oplus\theta(\mathfrak{s}_{m+i})$. Then, for $i\in \{1,\cdots ,m\}$,  $ \mathfrak{S}_i$ is a simple ideal of $\mathfrak{s}$ that is invariant by $\theta$. Also, for $i\in \{1,\cdots ,p\}$,  $ \mathfrak{S}_{m+i}$ is a semisimple ideal of $\mathfrak{s}$ that is invariant by $\theta$.

Since, for $x,y\in\g$,  $[x,y]=[\theta(x),\theta(y)]_\theta$ then it is clear that
\begin{itemize}
\item $\mathfrak{R}(\g)$ is a solvable ideal of the Hom-Lie algebra $(\g, [\ ,\ ],\theta)$.
\item $\mathfrak{s}$ is a subalgebra of the Hom-Lie algebra $(\g, [\ ,\ ],\theta)$.
\item For $i\in \{1,\cdots ,m+p\}$
\begin{enumerate}
\item $ \mathfrak{S}_i$ is a simple ideal of the Hom-Lie algebra $(\mathfrak{s},{ [\ ,\ ]_\theta}_{|\mathfrak{s}\times \mathfrak{s}},\theta _{|\mathfrak{s}})$ is a simple
\item $ \mathfrak{S}_i$ is a simple subalgebra of the Hom-Lie algebra $(\g, [\ ,\ ],\theta)$.
\end{enumerate}
\end{itemize}
\begin{proposition}
$\mathfrak{R}(\g)$ is the greatest solvable ideal of the Hom-Lie algebra $(\g, [\ ,\ ],\theta)$.
\end{proposition}
\begin{proof}Indeed, if $I$ is a solvable ideal of $(\g, [\ ,\ ],\theta)$, then $I$ is a solvable ideal of the Lie algebra $\g_\theta$, because $\theta(I)=I$. Consequently $I\subseteq \mathfrak{R}(\g)$.
\end{proof}
\begin{corollary}\label{cor6.5}Let $(\g, [\ ,\ ],\theta)$ be an involutive Hom-Lie algebra then
$$\g= \mathfrak{S}_1\oplus\cdots \mathfrak{S}_{m+p}\oplus  \mathfrak{R}(\g)$$
where
\begin{itemize}
\item for $i\in \{1,\cdots ,m+p\}$, $ \mathfrak{S}_i$ is a simple subalgebra of $(\g, [\ ,\ ],\theta)$,
\item $\mathfrak{R}(\g)$ is the greatest solvable ideal of $(\g, [\ ,\ ],\theta)$.
\end{itemize}
\end{corollary}

 \section{Double extension Theorems of Hom-Lie algebras}
 The  fundamental result on quadratic Lie algebras in \cite{MedinaRevoy} leads to  constructing and characterizing  quadratic Lie algebras   using double extension. While  $T^*$-extension concept  introduced in  \cite{Bordemann97}  concerns nonassociative algebras with nondegenerate associative symmetric bilinear form.

The following theorem extends the double extension by one-dimensional Lie algebras to  Hom-Lie algebras case.

 \begin{theorem}\label{DoubleExtension0}
 Let $(V,[\ ,\ ]_{V},\alpha_V,B_V)$ be a quadratic multiplicative Hom-Lie algebra and   $\delta: V\rightarrow V$ be a linear map. Set  $x_0\in V$, $\lambda,\lambda_0 \in\K$.  We denote, for $x\in V$,  $\rho_V(x):=[x,\cdot]_V$ and the bracket of two linear maps stands for the commutator on $End(V).$

 If hold the following conditions
 \begin{eqnarray}\label{DE1}
 && \alpha_V \circ \delta \circ \alpha_V-\lambda \delta=\rho_V(x_0),\\\label{DE2}
 &&  [\alpha_V,\delta^2]=\rho_V(\delta (x_0)),\\\label{DE3}
&& \forall x,y\in V, \ \  (\lambda\delta+\rho_V(x_0))([x,y]_V)=[\delta (x),\alpha_V(y)]_V+[\alpha_V(x),\delta (y)]_V,
 \end{eqnarray}
 then the vector space $\g:=\K b\oplus V \oplus\K e$, where $\K b$  and  $\K e$ are  one-dimensional vector spaces, is a multiplicative Hom-Lie algebra with the bracket $[\ ,\ ]:\g \times \g \rightarrow\g$ defined by
 \begin{align*}
 [x,y]&=[x,y]_V+B_V(\delta(x),y)\quad \forall x,y\in V\\
 [b,x]&=\delta (x) \quad \forall x\in V\\
 [e,z]&=0\quad \forall z\in\g
 \end{align*}
 and the linear map $\alpha:\g \rightarrow \g$ defined by
  \begin{align*}
\alpha (x)&=\alpha_V(x)+ B_V(x_0,x) e\quad \forall x\in V \\
\alpha (b)&=\lambda b +x_0+\lambda_0 e \\
\alpha (e)&=\lambda e
 \end{align*}

 In addition the symmetric bilinear form $B:\g \times \g \rightarrow\K$ defined by
 \begin{align*}
 B(x,y)&=B_V(x,y)\quad \forall x,y\in V\\
 B(b,e)&=1\\
 B(x,b)&=B(x,e )=0\quad \forall x\in V\\
 B(b,b)&=B(e, e)=0
 \end{align*}
 is an invariant scalar product on $(\g, [\ ,\ ],\alpha )$.

 In the particular case,  when $\alpha_V$ is invertible then $\alpha$ is invertible if and only if $\lambda\neq 0$ and when $\alpha_V$ is an involution, then $\alpha$ is an involution if and only if
 \begin{equation}\label{DE4}
 \lambda \in\{-1,1\},\ \  \alpha(x_0)=-\lambda x_0\  \ \text{ and  } \lambda_0=-\frac{1}{2\lambda}B_V(x_0,x_0).
 \end{equation}

 \end{theorem}
\begin{proof}
 One shows that $\g$ is a Hom-Lie algebra and $B$  an invariant scalar product  by straightforward calculations.

We assume now that $\alpha$ is invertible. Then $\lambda\neq 0.$ Let $x\in V$ such that $\alpha_V(x)=0.$ It follows that
$\alpha(x)=B_V(x_0,x)e$ which implies   $\alpha(x)=\lambda^{-1}B_V(x_0,x)\alpha(e).$ Then $x=\lambda^{-1}B_V(x_0,x)e$. Thus
$ \lambda^{-1}B_V(x_0,x)=0$, then $B_V(x_0,x)=0.$  Likewise $x=0.$ Hence $\alpha$ is invertible.

Conversely, assume that $\alpha_V$ is invertible and  $\lambda\neq 0$.  Let $x\in\g$ such that $\alpha(x)=0$. The element $x$ in $\g$ may be written $x=r b+y+ s e$ where $r,s\in\K$ and $y\in V.$ Then
\begin{align*}
0=\alpha(x)& =r(\lambda b+x_0+ \lambda_0 e)+(\alpha_V(y)+B_V(x_0,y)e)+ s \lambda e\\
\ & =(\lambda r )b +(\alpha_V(y) +r x_0)+(r \lambda_0+B(x_0,y) + s \lambda)e.
\end{align*}
It follows
\begin{eqnarray*}
\lambda r=0\\
\alpha_V(y) =-r x_0\\
r \lambda_0+B_V(x_0,y)+s h=0.
\end{eqnarray*}
The solution of the system is $r=0,\  \alpha_V(y) =0,\  B_V(x_0,y)e)=- s \lambda.$ Then $r=0,\  y =0,\  s=0,$ which gives $x=0$.
Then we conclude that $\alpha$ is invertible.

Now, assume that  $\alpha$ is an involution. Then $b=\alpha^2(b)$ is equivalent to
\begin{align*}
b &= \lambda \alpha(b) +\alpha(x_0)+\lambda \alpha(e),\\
\ & =\lambda (\lambda b+x_0+\lambda_0 e)+\alpha_V(x_0)+B(x_0,x_0)e+\lambda_0 \lambda e
\end{align*}
By identification, it follows
$$\{ \begin{array}{c}
    1-\lambda^2 =0 \\
    \alpha_V (x_0)=-\lambda x_0 \\
    B(x_0,x_0)=-2\lambda \lambda_0.
  \end{array}\
$$
Therefore, $\lambda=\pm 1$ which are the eigenvalues of the involution, $\alpha_V (x_0)=-\lambda x_0$ and
 $\lambda_0=-\frac{1}{2 \lambda} B_V(x_0,x_0).$

 Let $x\in V$. The identity $\alpha^2(x)=x$ may be written
 \begin{align*}
 \alpha (\alpha_V(x)+B(x_0,x)e)=x\\
  \alpha (\alpha_V(x))+B(x_0,x)\alpha(e)=x\\
  \alpha^2 (x)+B(x_0,\alpha_V(x))e+\lambda B(x_0,x)e=x
 \end{align*}
 Then its equivalent to
 $\alpha_V^2(x)=x$ and $ \alpha_V(x_0)=-\lambda x_0.$
This ends the proof that $\alpha$ is an involution.

Moreover,  we have proved that $\alpha$ is an involution is equivalent to
\begin{eqnarray*}
\alpha_V^2=id_V, \quad \lambda^2=\pm 1,\quad \lambda_0=-\frac{1}{2 \lambda} B_V(x_0,x_0).
\end{eqnarray*}
\end{proof}

\begin{definition}
 The quadratic multiplicative Hom-Lie algebra $(\g, [\ ,\ ],\alpha ,B)$ constructed in Theorem \ref{DoubleExtension0} is called a double extension of $(V,[\ ,\ ]_{V},\alpha_V,B_V)$ ( by the one-dimensional Lie algebra) by means $(\delta,x_0,\lambda, \lambda_0)$.

 It is called involutive double extension when  $\alpha_V$ is an involution and the condition \eqref{DE4} satisfied and its called regular double extension when  $\alpha_V$ is invertible and $\lambda\neq 0.$
 \end{definition}

 We provide in the following the converse of the previous Theorem.

 \begin{theorem}\label{ConvThDE}
Let  $(\g, [\ ,\ ],\alpha ,B)$  be an irreducible quadratic multiplicative Hom-Lie algebra such that $dim\g >1.$

If $\mathcal{Z}(\g )\neq \{0\},$ then $(\g, [\ ,\ ],\alpha ,B)$ is a double extension of a quadratic Hom-Lie algebra $(V,[\ ,\ ]_{V},\alpha_V,B_V)$ (by a one-dimensional Lie algebra) such that $dimV= dim\g -2.$

If in addition $\alpha$ is an involution then the  IQH-Lie algebra $(\g, [\ ,\ ],\alpha ,B)$ is an involutive double extension of an a IQH-Lie algebra $(V,[\ ,\ ]_{V},\alpha_V,B_V)$ and if $\alpha$ is invertible then the RQH-Lie algebra $(\g, [\ ,\ ],\alpha ,B)$ is a regular double extension of  a RQH-Lie algebra $(V,[\ ,\ ]_{V},\alpha_V,B_V).$
 \end{theorem}
 \begin{proof}
Let $(\g, [\ ,\ ],\alpha ,B)$ be an irreducible quadratic Hom-Lie algebra such that $dim\g >1$. Assume that $\mathcal{Z}(\g) \neq \{0\}$ and $\alpha (\mathcal{Z}(\g))\subseteq \mathcal{Z}(\g).$ Then there exists $\lambda\in\K$ ($\K$ algebraically closed) and $e\in\mathcal{Z}(\g)\setminus \{0\}$ such that $\alpha(e)=\lambda e$. We have  $B(e,e)=0$ because $\K e$ is an ideal of $\g$, $\g$ is irreducible and $dim\g > 1$. Since $B$ is nondegenerate then there exists $b\in\g$ such that $B(e,b)=1$ and $B(b,b)=0.$

The vector space $A=\K e\oplus \K b$ is nondegenerate (i.e. $B_{|A\times A}$ is nondegenerate ), then $\g =\A \oplus \A^{\perp}.$  Set  $V=\A^{\perp}$, then $B_{|V\times V}$ is nondegenerate and $(\K e)^{\perp}=\K e \oplus V.$

Then $\g =\K e \oplus V\oplus  \K b$ with $B(e,e)=B(b,b)=0$, $B(e,b)=1$, $B(V,e)=B(V,b)=\{0\}$ and  $B_{|V\times V}$ is nondegenerate.

There exist  a linear map  $\alpha_V: V\rightarrow V$  and a linear form $f: V\rightarrow \K $ such that $\alpha (x)=f(x)e+\alpha_V(x)$, $\forall x\in V$, because $(\K e)^{\perp}$ is an ideal of $\g$ implies $\alpha((\K e)^{\perp})\subseteq (\K e)^{\perp}.$

There exist $\lambda_0 \in \K,$ $\gamma_0\in\K$, $x_0\in W$ such that $\alpha (b)=\lambda_0 e+x_0+\gamma_0 b.$

The fact that $(\K e)^{\perp}$ is ideal implies that
\begin{enumerate}
\item there exist  a bilinear map $[\ ,\ ]_V:V\times V\rightarrow V$ and a bilinear form $\varphi:V\times V\rightarrow \K$ such that
\begin{equation}\label{M1}
[x,y]=\varphi(x,y)e+[x,y]_V, \ \ \forall x,y\in V.
\end{equation}
\item there exist  a linear map $\delta:V\rightarrow V$ and a linear form $h:V\rightarrow \K$ such that
\begin{equation}\label{M2}
[b,y]=\delta(x)+h(x)e, \ \ \forall x\in V.
\end{equation}
\end{enumerate}
Let $x,y,z\in V$, we have $[\alpha(x),[y,z]]=[\alpha(x),[y,z]_V]$ because $e\in \mathcal{Z}(\g)$. Then
$$
[\alpha(x),[y,z]]=[\alpha_V(x),[y,z]_V]=[\alpha_V(x),[y,z]_V]_V+\varphi(\alpha_V(x),[y,z]_V)e.$$
Therefore the Hom-Jacobi identity $ \circlearrowleft_{x,y,z}{[\alpha(x),[y,z]]}=0 $ implies that
\begin{equation}\label{I}
\circlearrowleft_{x,y,z}{[\alpha_V(x),[y,z]_V]_V}=0,
\end{equation}
and
\begin{equation}\label{II}
\circlearrowleft_{x,y,z}{\varphi(\alpha_V(x),[y,z]_V)}=0.
\end{equation}
Also the skewsymmetry $[\ ,\ ]$ is equivalent for $x,y\in V$  to
\begin{eqnarray}
\label{etoile}[x,y]_V=-[y,x]_V,\\
\label{etoilePrime}\varphi (x,y)=-\varphi (y,x).
\end{eqnarray}
Thus \eqref{I} and \eqref{etoile} show that $(V,[\ ,\ ]_V,\alpha_V)$ is a Hom-Lie algebra.


Let $x,y\in\mathcal{Z}$
\begin{align*}
0& =[\alpha(b),[x,y]]+[\alpha(x),[y,b]]+[\alpha(y),[b,x]]\\
\ & =[x_0,[x,y]]+\lambda [b,[x,y]]+[\alpha_V(x),[y,b]]+[\alpha_V(y),[b,x]]\\
\ & =[x_0,[x,y]_V]+\lambda [b,[x,y]_V]-[\alpha_V(x),\delta(y)]+[\alpha_V(y),\delta(x)]\\
\ & =[x_0,[x,y]_V]_V+\varphi(x_0,[x,y]_V)e+\lambda \delta([x,y]_V)+\lambda  h([x,y]_V)e\\
\ \ & +[\alpha_V(x),\delta(y)]_V +\varphi(\alpha_V(x),\delta(y))e+[\alpha_V(y),\delta(x)]_V
+\varphi(\alpha_V(y),\delta(x))e.\\
\end{align*}
Therefore
\begin{eqnarray}
\label{III}
\varphi(x_0,[x,y]_V)+\lambda h([x,y]_V)\
  -\varphi(\alpha_V(x),\delta(y))
+\varphi(\alpha_V(y),\delta(x))=0,\\ \label{IV}
[x_0,[x,y]_V]_V+\lambda \delta([x,y]_V)-[\alpha_V(x),\delta(y)]_V+[\alpha_V(y),\delta(x)]_V=0.
\end{eqnarray}
Let $x,y\in V$
$$\varphi(x,y)=B([x,y],b)=B(b,[x,y])=B([b,x],y)=B(\delta(x),y),$$
because $B(e,y)=0.$ Then $[x,y]=[x,y]_V+B(\delta(x),y)e,$ that is \eqref{M1}.

Let $x\in V$,
$$B([b,x],b)=-B(x,[b,b])=0.$$
In the other hand $B([b,x],b)=h(x)$ because $B(V,b)=\{0\}$ and $B(e,b)=1.$ Then $h(x)=0.$ We conclude that $[B,x]=\delta(x)\in V$ i.e. \eqref{M2}.

Let $y\in V$, $B(\alpha(b),y)=B(b,\alpha(y)),$ which is equivalent to $B(x_0,y)=B(b,\alpha_V(y)+f(y)e)=f(y).$ Then
$f(y)=B(x_0,y).$ It follows $\alpha(y)=\alpha_V(y)+B(x_0,y)e.$

Also $B(\alpha(b),e)=B(b,\alpha(e))$ is equivalent to $\gamma_0=B(b,\lambda e)=\lambda.$ Then
$$\alpha(b)=\lambda_0e+x_0+\lambda b=\lambda b+x_0+\lambda_0 e.$$

Now \eqref{III} is equivalent to
\begin{equation}\label{IIIPrime}
B(\delta(x_0),[x,y]_V)-B(\delta(\alpha_V(x)),\delta(y))+B(\delta(\alpha_V(y)),\delta(x))=0,
\end{equation}
because $h=0$ and $\varphi(x,y)=B(\delta(x),y)$, $\forall x,y\in V.$

Let $x,y,z\in V$
\begin{eqnarray*}
B([x,y],z)=B([x,y]_V,z)=B_V([x,y]_V,z),\\
B(x,[y,z])=B(x,[y,z]_V)=B_V(x,[y,z]_V).
\end{eqnarray*}
Then $B_V$ is invariant.

In addition $B(\alpha(x),y)=B(x,\alpha(y))$ implies $B_V(\alpha_V(x),y)=B_V(x,\alpha_V(y)),$ that is $\alpha_V$ is $B_V$-symmetric. It is obvious  that $B_V$ is symmetric.

We have
\begin{align*}
\alpha([x,y])&=\alpha([x,y]_V+B(\delta(x),y)e)\\
\ &=\alpha_V([x,y]_V)+B(x_0,[x,y]_V)e+\lambda B(\delta(x),y)e\\
\ &=\alpha_V([x,y]_V)+B_V(x_0,[x,y]_V)e+ B(\lambda\delta(x),y)e.
\end{align*}
In the other hand
$$[\alpha(x),\alpha(y)]=[\alpha_V(x),\alpha_V(y)]_V+B(\delta(\alpha_V(x)),\alpha_V(y))e.$$
Then
\begin{equation}\label{At1}
\alpha_V([x,y]_V)=[\alpha_V(x),\alpha_V(y)].
\end{equation}
Also
$$B_V(x_0,[x,y]_V)+B(\lambda \delta(x),y)=B(\delta(\alpha_V(x),\alpha_V(y))$$
which is equivalent to
$$B_V([x_0,x]_V,y)+B(\lambda \delta(x),y)=B(\alpha_V\circ\delta\circ\alpha_V(x),y),$$
and to
\begin{equation}\label{BB1}
B_V([x_0,x]_V+\lambda \delta(x)-\alpha_V\circ\delta\circ\alpha_V(x),y)=0.
\end{equation}
Since $B$ is nondegenerate, then \eqref{BB1} leads to
$$[x_0,x]_V+\lambda \delta(x)-\alpha_V\circ\delta\circ\alpha_V(x)=0.$$
By setting $\rho_V(x_0)=[x_0,\cdot]_V$, the previous identity may be written
\begin{equation}\label{CDC1}
\alpha_V\circ\delta\circ\alpha_V-\lambda \delta(x)=\rho_V(x_0).
\end{equation}

Then we have proved that $(V,[\ ,\ ]_{V},\alpha_V,B_V)$ is quadratic Hom-Lie algebra. Also there exists $\delta\in End(V)$ such that the condition \ref{CDC1} is satisfied. Hence
\begin{align*}
 [x,y]&=[x,y]_V+B_V(\delta(x),y)\quad \forall x,y\in V\\
 [b,x]&=\delta (x) \quad \forall x\in V\\
 \end{align*}
 and the linear map $\alpha:\g \rightarrow \g$ defined by
  \begin{align*}
\alpha (x)&=\alpha_V(x)+ B_V(x_0,x) e\quad \forall x\in V \\
\alpha (e)&=\lambda e  \\
\alpha (b)&=\lambda b +x_0+\lambda_0 e.
 \end{align*}

Let $x,y\in V$, we have $B([b,x],y)=-B(x,[b,y])$ which is equivalent to $B(\delta(x),y)=-B(x,\delta(y)),$ then to $B_V(\delta(x),y)=-B(x,\delta(y)).$ It means that $\delta$ is skewsymmetric with respect to $B$.

The condition \eqref{IIIPrime} is equivalent to $B_V(\delta(x_0),[x,y]_V)-B-V(\delta\circ\alpha_V(x),\delta(y))+B_V(\delta\circ\alpha_V(y),\delta(x))=0$ then $B_V([\delta(x_0),x],y)+B(\delta^2\circ\alpha_V(x),y)-B_V(y,\alpha_V\circ\delta^2(x))=0$. It is equivalent to $[\delta(x_0),x]+(\delta^2\circ\alpha_V-\alpha_V\circ\delta^2)(x)=0$, thus to
\begin{equation}\label{CDC2}
\delta^2\circ\alpha_V-\alpha_V\circ\delta^2=-[\delta(x_0),\cdot].
\end{equation}
Recall that we have denoted $[\delta(x_0),\cdot]$ by $\rho(\delta(x_0)).$

Let $x,y\in V$. We have
\begin{align*}
\circlearrowleft_{b,x,y}{[\alpha(b),[x,y]]}&=[\lambda b+x_0,[x,y]_V]-[\alpha_V(x),\delta(y)]+[\alpha_V(y),\delta(x)]\\
\ &=\lambda \delta([x,y]_V)+[x_0,[x,y]_V]_V+B(\delta(x_0),[x,y]_V)e-[\alpha_V(x),\delta(y)]_V\\
\ & -B(\delta(\alpha_V(x)),\delta(y))e
+[\alpha_V(y),\delta(x)]_V+B(\delta(\alpha_V(y)),\delta(x))e=0
\end{align*}
which leads to
\begin{equation*}\lambda \delta([x,y]_V)+[x_0,[x,y]_V]-[\alpha_V(x),\delta(y)]_V+[\alpha_V(y),\delta(x)]_V=0\end{equation*}
and
\begin{equation*}B([\delta(x_0),x]+\delta^2\circ\alpha_V(x)-\alpha_V\circ\delta^2(x),y)=0.\end{equation*}
These two   identities are equivalent to
\begin{equation*}\lambda \delta([x,y]_V)=[\delta(x),\alpha_V(y)]_V+[\alpha_V(x),\delta(y)]_V-[x_0,[x,y]_V]\end{equation*}
which is equivalent to \eqref{IIIPrime} and
\begin{equation*}
\delta^2\circ \alpha_V-\alpha_V\circ \delta^2=[\delta(x_0),\cdot ],
\end{equation*}
which is equivalent to  \eqref{CDC2}.
The identity \eqref{eq3} may be written
\begin{equation}\label{CDC3}
(\lambda \delta +[x_0,\cdot ]_V)([x,y]_V)=[\delta(x),\alpha_V(y]_V+[\alpha_V(x),\delta(y)]_V.
\end{equation}
$\circlearrowleft_{b,b,x}{[\alpha(b),[b,x]]}=0
$ is satisfied.

We have proved that $(V,[\ ,\ ]_{V},\alpha_V,B_V)$ is quadratic multiplicative Hom-Lie algebra. Also it is clear  that using identities \eqref{CDC1},\eqref{CDC2},\eqref{CDC3} the Hom-Lie algebra $\g$ is a double extension of this Hom-Lie algebra by means of $(\delta, x_0,\lambda, \lambda_0)$.
%
  \end{proof}
 Theorem \ref{ConvThDE} and the Lemma \ref{CentreNul} lead to
 \begin{proposition}\label{prop7.4}
 Let $(\g, [\ ,\ ],\alpha ,B)$ be a RQH-Lie algebra. Then it is obtained from a centerless  IQH-Lie algebra by a finite sequence of regular  double extensions by the one dimensional Lie algebra or/and orthogonal direct sum of RQH-Lie algebras.
\end{proposition}

%

\begin{corollary}
 Let $(\g, [\ ,\ ],\alpha ,B)$ be an IQH-Lie algebra.
 Then it is obtained from  centerless IQH-Lie algebras  by a finite sequence of involutive  double extensions by the one dimensional Lie algebra or/and orthogonal direct sum of IQH-Lie algebras.
\end{corollary}
\begin{remark}
The Proposition \ref{prop7.4}  reduces the study of IQH-Lie algebras to centerless IQH-Lie algebras.

Hence, the problem of Quadratic Hom-Lie algebras reduces to study two classes
\begin{enumerate}
\item Quadratic Hom-Lie algebras $(\g, [\ ,\ ],\alpha ,B)$  with $\alpha$ nilpotent,
\item centerless IQH-Lie algebras.
\end{enumerate}
\end{remark}

We study in the following the double extension of IQH-algebra  by an involutive Hom-Lie algebra of any dimension.
\begin{theorem}[Involutive double extension Theorem]
 Let $(V,[\ ,\ ]_{V},\alpha_V,B_V)$ be an IQH-Lie algebra and   $(A,[\ ,\ ]_A,\alpha_A)$ be an involutive Hom-Lie algebra.

 Let $\phi :\begin{array}{c}
              A\rightarrow End(V) \\
              a\rightarrow \phi (a)
            \end{array}
 $ be a representation of $A$ on $ (V,\alpha_A)$ such that

 \begin{eqnarray}\label{TDE1}
 && \alpha_V \circ \phi(a)([x,y]_V)=[\phi (a)\circ\alpha_V(x),y]_V+[x,\phi(a) \circ\alpha_V (y)]_V \  \  \forall x,y\in V,\\\label{TDE2}
 && \phi (a)\circ\alpha_A(a)=\alpha_V\circ\phi (a)\circ\alpha_V\  \ \cdot\forall a\in A ,\\\label{TDE3}
&& B_V(\phi (a)(x),y)=-B_V(x,\phi (a)(y) \  \  \forall x,y\in V \ \  \  \forall a\in A.
 \end{eqnarray}

 Let $\psi :\begin{array}{cc}
              V\times V& \rightarrow A^\ast \\
              (x,y)& \rightarrow \psi (x,y)
            \end{array}
 $ defined by $\forall x,y\in V \    \forall a\in A$ by $\psi(x,y)(a)=B_V(\phi (a)(x),y).$

 Then the vector space $\g:=A\oplus V \oplus A^\ast $   endowed with  the bracket $[\ ,\ ]:\g \times \g \rightarrow\g$ defined for $f,f'\in A^\ast$, $a,a'in A$ and $v,v'in V$ by
 \begin{equation*}
 [f+v+a,f'+v'+a']=\widetilde{\ad}_A(a)(f')-\widetilde{\ad}_A(a')(f)+\psi (v,v')+[v,v']_V+[a,a']_A
 \end{equation*}
 and the linear map $\alpha:\g \rightarrow \g$ defined by
  \begin{equation*}
\alpha (f+v+a)=^t \alpha_A(a)+\alpha_V (f)+\alpha_A (a)
 \end{equation*}
is an involutive  Hom-Lie algebra.

 In addition, the bilinear form $B:\g \times \g \rightarrow \K$ defined by
 \begin{equation*}
 B_\gamma(f+v+a,f'+v'+a')=B_V(a,a')+f(w')+f'(w)+\gamma (a,a')
 \end{equation*}
 where $\gamma$ is an invariant symmetric bilinear form on $(A,[\ ,\ ],\alpha_A )$ such that $\gamma(\alpha(a),a')=\gamma(a,\alpha(a'))$,  $\forall a,a'\in A$, is a quadratic structure on $\g, [\ ,\ ],\alpha).$

 The IQH-Lie algebra  $\g, [\ ,\ ],\alpha,B_\gamma)$ is called the involutive double extension of $(V,[\ ,\ ]_{V},\alpha_V,B_V)$ by $(A,[\ ,\ ]_A,\alpha_A)$ by means of $(\phi,\gamma).$
 \end{theorem}

%
\begin{proof}
The proof is similar to the proof in \cite{BenamorBenayadi}.
\end{proof}


We pursue the characterization of IQH-Lie algebras.
 \begin{theorem}
 Let  $( \g, [\ ,\ ],\alpha,B)$ be an irreducible   IQH-Lie algebra such that $\mathcal{Z}(\g)=\{0\}$. Then  it is either an involutive simple Hom-Lie algebra or it is obtained by an involutive double extension of an IQH-Lie algebra $(V,[\ ,\ ]_{V},\alpha_V,B_V)$ by an involutive simple Hom-Lie algebra.
 \end{theorem}
\begin{proof}
Let  $( \g, [\ ,\ ],\alpha,Q)$ be an irreducible   involutive Hom-Lie algebra such that  $\g$ is not simple, $dim\g >1$ and $\mathcal{Z}(\g)=\{0\}$.
Then $\g \neq \mathfrak{R}(\g)$ and $\g=\mathfrak{s}_1\oplus\cdots\oplus \mathfrak{s}_n \oplus\mathfrak{ R}(\g)$ where $\mathfrak{s}_i$ are simple Hom-subalgebras (Corollary \ref{cor6.5}), $J=\mathfrak{s}_2\oplus\cdots\oplus \mathfrak{s}_n \oplus\mathfrak{ R}(\g)$ is a maximal ideal of $\g$ such that $\g=\mathfrak{s}_1\oplus J$.
Set $A=\mathfrak{s}_1$ which  is a subalgebra of  $( \g, [\ ,\ ],\alpha)$. The  fact that $J$ is a maximal ideal implies that $I:=J^\bot$ is a minimal ideal of $\g$. Also $I\bigcap I^\bot$ is an ideal of $\g$ contained in $I$. Then $I\bigcap I^\bot=\{ 0\}$ or $I\bigcap I^\bot=I$, i.e. $I\subset I^\bot=J$. Since $\g$ is irreducible then it implies that $I\subset J$, consequently $B(I,I)$=\{0\}, i.e. $I$ is isotropy.

It is clear that $B_{|(I\oplus A)\times (I\oplus A)}$ is nondegenerate. It follows that $\g=(I\oplus A)\oplus (I\oplus A)^\bot$.
We set $V:=(I\oplus A)^\bot$, which is a vector subspace of $\g$.
Consequently, $J=I\oplus V$ and $B_{V\times V}$ is nondegenerate. Since $[I,I^\bot ]=\{0\}$, then  $[I,I]=\{0\}$ because $I\subset J=I^\perp $ and $[I,V]=\{ 0\}$.

Let $x,y\in V$, then $[x,y]=[x,y]_V+\varphi (x,y)$, where $[\ ,\ ]_V:V\times V\rightarrow V$ and  $\varphi:V\times V\rightarrow I$ are bilinear maps.

Let $x\in V$ and $a\in A$, then $[a,x]=\phi(a)(x)+\varphi'(a,x)$ where  $\varphi':A\times V\rightarrow I$ is a bilinear map and $\phi(a)\in End(V)$.

We denote, for $a,a'\in A$, the bracket  $[a,a']$ by $[a,a']_A$.

Let $x,y\in V$ and $a\in A$, we have $B([x,y],a)=B(\varphi(x,y),a)$ and $B(x,[y,a])=-B(x,[a,y])=-B(x,\phi(a)y)$, then
$$B(\varphi(x,y),a)=-B(x,\phi(a)y).$$
Also $B([a,x],y)=-B(x,[a,y])$ is equivalent to
$$B(\phi(a)x,y)=-B(x,\phi(a)y).$$
Therefore we have
$$B(\varphi(x,y),a)=-B(x,\phi(a)y)=B(\phi(a)x,y).$$

For $x\in V$ we have  $\alpha (x)\in J=I^\perp$ which implies $\alpha (x)=\alpha_V(x)+\gamma (x)$ where $\alpha_V(x)\in V$ and $\gamma (x)\in I.$

Let  $x\in V$ and $a\in A$, then $B(\alpha (x),a)=B(x,\alpha(a))$ implies $B(\gamma (x), a)=B(x,\alpha (a))=0$ since $B(V,A)=\{ 0\}$. Therefore $\gamma (x)=0.$  Thus $\alpha(V)\subseteq V$ and we denote in the sequel $\alpha_{| V}$ by $\alpha_V$.

Let  $x\in V$ and $a,b\in A$, then $B([a,x],b)=B(\varphi'(a,x),b)$ and on the other hand $B([a,x],b)=-B(x,[a,b])=0$ (because $B(V,A)=\{ 0\}$). Therefore $B(\varphi'(a,x),b)=0$ for any $b\in A$. Then $\varphi'(a,x)=0.$ Thus
$$[a,x]=\phi(a)x, \quad \forall a\in A, \forall x\in V.
$$
Let  $x,y\in V$, then $[x,y]=-[y,x]$ implies
\begin{eqnarray*}
[x,y]_V=-[y,x]_V,\\
\varphi (x,y)=-\varphi (y,x).
\end{eqnarray*}
We show that $[\ ,\ ]_V$ and $\varphi$ satisfy the Hom-Jacobi condition. Let $x,y,z\in V$, we have
 \begin{align*}
[\alpha (x),[y,z]]&= [\alpha (x),[y,z]_V] \quad \text{because } [V,I]=\{0\}\\
\ & = [\alpha (x),[y,z]_V]_V+ \varphi(\alpha (x),[y,z]_V)\\
\ & = [\alpha_V (x),[y,z]_V]_V+ \varphi(\alpha_V (x),[y,z]_V).
\end{align*}
Then
$$0=\circlearrowleft_{x,y,z}{[\alpha (x),[y,z]]}=\circlearrowleft_{x,y,z}{[\alpha_V (x),[y,z]_V]_V}+\circlearrowleft_{x,y,z}{\varphi(\alpha_V (x),[y,z]_V)}
$$
Since the first summation is in $V$ and the second in $I$ then we obtain
\begin{eqnarray}
\circlearrowleft_{x,y,z}{[\alpha_V (x),[y,z]_V]_V}=0,\\\label{HomJacobiW}
\circlearrowleft_{x,y,z}{\varphi(\alpha_V (x),[y,z]_V)}=0.
\end{eqnarray}
Also
$\alpha([x,y])=[\alpha(x),\alpha(y)]=[\alpha_V(x),\alpha_V(y)]$ implies
$\alpha_V([x,y]_V)+\alpha(\varphi(x,y))=[\alpha_V(x),\alpha_V(y)]_V+\varphi(\alpha_V(x),\alpha_V(y)).$ Therefore
\begin{eqnarray*}
\alpha_V([x,y]_V)=[\alpha_V(x),\alpha_V(y)]_V\\
\alpha(\varphi(x,y))=\varphi(\alpha_V(x),\alpha_V(y)).
\end{eqnarray*}


It is obvious that  $\alpha^2=id$ induces $\alpha_V^2=id$.

The previous calculations leads to the fact that $(V,[\ ,\ ]_V,\alpha_V)$ is an involutive Hom-Lie algebra.

The Hom-Lie algebra $(V,[\ ,\ ]_V,\alpha_V)$ is an IQH-Lie algebra by setting $B_V=B_{|V\times V}$. Indeed for any $x,y,z\in V$ we have $$B([x,y]_V,z)=B([x,y],z)=B(x,[y,z])=B(x,[y,z]_V)$$
and $$B(\alpha_V(x),y)=B(\alpha(x),y)=B(x,\alpha(y))=B(x,\alpha_V(y)).$$

Using the identity \eqref{HomJacobiW} we have for any $a\in A$ and $x,y,z\in V$
\begin{eqnarray*}
B(\circlearrowleft_{x,y,z}{\varphi(\alpha_V (x),[y,z]_V)},a)=0\\
\Longleftrightarrow\ \circlearrowleft_{x,y,z}{B(\varphi(\alpha_V (x),[y,z]_V),a)}=0\\
\Longleftrightarrow\ \circlearrowleft_{x,y,z}{B(\alpha_V (x),\phi(a)([y,z]_V))}=0\\
\end{eqnarray*}
then we have
\begin{equation}\label{star}
\circlearrowleft_{x,y,z}{B(\phi(a)(\alpha_V (x)),[y,z]_V)}=0
\end{equation}

Expanding the identity \eqref{star} and using the fact that $(V,[\ ,\ ]_V,\alpha_V,B_V)$ is an IQH-Lie algebra we obtain
\begin{equation*}
B_V(\phi(a)(\alpha_V (x)),[y,z]_V)+B_V(\phi(a)(\alpha_V (y)),[z,x]_V)+B_V(\phi(a)(\alpha_V (z)),[x,y]_V)=0
\end{equation*}
which is equivalent to
\begin{equation*}
B_V([\phi(a)\circ\alpha_V (x),y]_V+[x,\phi(a)\circ\alpha_V (y)]-\alpha_V\circ\phi(a)[x,y]_V,z)=0.
\end{equation*}
Then we have
\begin{equation}\label{plus1}
\alpha_V\circ\phi(a)[x,y]_V=[\phi(a)\circ\alpha_V (x),y]_V+[x,\phi(a)\circ\alpha_V (y)]_V.
\end{equation}

Let $a\in A$ and $x,y\in V$, the Hom-Jacobi identity
\begin{equation*}
[\alpha(a),[x,y]]+[\alpha(x),[y,a]]+[\alpha(y),[a,x]]=0
\end{equation*}
may be written
\begin{equation*}
[\alpha(a),[x,y]_V]+[\alpha(a),\varphi(x,y)]-[\alpha(x),\phi(a)y]_V-\varphi(\alpha(x),\phi(a)y)+[\alpha(y),\phi(a)x]_V+\varphi(\alpha(y),\phi(a)x)=0.
\end{equation*}
It's equivalent to
\begin{equation*}
\phi(\alpha(a))([x,y]_V)+[\alpha(a),\varphi(x,y)]-[\alpha(x),\phi(a)y]_V-\varphi(\alpha(x),\phi(a)y)+[\alpha(y),\phi(a)x]_V+\varphi(\alpha(y),\phi(a)x)=0.
\end{equation*}
By gathering the elements in $V$ and elements in $I$, the previous identity leads to
\begin{eqnarray}\label{plus2}
\phi(\alpha(a))([x,y]_V)=[\alpha(x),\phi(a)y]_V+[\phi(a)x,\alpha(y)]_V,\\
\label{star2} [\alpha(a),\varphi(x,y)]=\varphi(\alpha(x),\phi(a)y)+\varphi(\phi(a)x,\alpha(y)).
\end{eqnarray}
The identity \eqref{plus2} leads for any $b\in A$ to
\begin{eqnarray*}
B(\phi(\alpha(a))([x,y]_V),b)=B([\alpha(x),\phi(a)y]_V,b)+B([\phi(a)x,\alpha(y)]_V,b)\\
\Leftrightarrow\ B(\varphi(x,y),[b,\alpha(a)])=B(\phi(b)(\alpha(x)),\phi(a)y)+B(\phi(b)\phi(a)(x),\alpha(y))\\
\Leftrightarrow \ B(\phi([b,\alpha(a)])(x),y)=-B(\phi(a)\phi(b)(\alpha(x)),y)+B(\alpha(\phi(b)\phi(a)(x)),y).
\end{eqnarray*}
Hence
\begin{equation*}\phi([b,\alpha(a)])(x)=-\phi(a)\phi(b)\alpha+\alpha(\phi(b)\phi(a),
\end{equation*}
which may be written
\begin{equation}\label{Z1}
\phi([\alpha(a),b])=\phi(a)\phi(b)\alpha-\alpha\phi(b)\phi(a).
\end{equation}
Let $a,b\in A$ and $x\in V$, the Hom-Jacobi identity
\begin{equation*}
[\alpha(a),[b,x]]+[\alpha(b),[x,a]]+[\alpha(x),[a,b]]=0
\end{equation*}
may be written
\begin{equation*}
\phi(\alpha(a))\phi(b)(x)-\phi(\alpha(b))\phi(a)(x)-\phi([a,b])(\alpha(x))=0.
\end{equation*}
Therefore we have
\begin{equation}\label{Z2}
\phi([a,b])\circ\alpha=\phi(\alpha(a))\phi(b)-\phi(\alpha(b))\phi(a).
\end{equation}
In fact the identities \eqref{Z1} and \eqref{Z2} are equivalent.

Let $a\in A$ and $x\in V$, the identity
$
\alpha([a,x])=[\alpha(a),\alpha(x)]
$
may be written $\alpha(\phi(a)(x))=\phi(\alpha(a))(\alpha(x))$ and also $\alpha_V(\phi(a)(x))=\phi(\alpha_A(a))(\alpha_V(x))$ where $\alpha_A=\alpha_{|A}$. Then we have
\begin{equation*}\phi\circ\alpha_A=\alpha_V\phi(a)\alpha_V^{-1}=\alpha_V\phi(a)\alpha_V.
\end{equation*}

Now, we evaluate $\phi (\alpha(a))[x,y]_V$ for $a\in A$ and $x,y\in V$. We have $\phi (\alpha(a))[x,y]_V=\alpha_V\circ (\alpha_V\circ\phi (\alpha(a))[x,y]_V$ since $\alpha_V$ is a involution. By applying the identity \eqref{plus1} we obtain
\begin{eqnarray*}
\phi (\alpha(a))[x,y]_V&&=\alpha_V([\phi(\alpha(a))\circ\alpha_V (x),y]_V+[x,\phi(\alpha(a))\circ\alpha_V (y)]_V)\\
\ &&=[\alpha_V\circ\phi(\alpha(a))\circ\alpha_V (x),\alpha_V(y)]_V+[\alpha_V(x),\alpha_V\circ\phi(\alpha(a))\circ\alpha_V (y)]_V\\
\ &&=[\alpha(x),\phi(a)y]_V+[\phi(a)x,\alpha(y)]_V
\end{eqnarray*}
Therefore we proved that the identity \eqref{plus1} implies \eqref{plus2}. In an other hand we have
\begin{eqnarray*}
\alpha_V \phi (a)[x,y]_V&&=\alpha_V(\alpha_V\phi(\alpha(a))\alpha_V)[x,y]_V\\
\ && = \phi(\alpha(a))\alpha_V[x,y]_V=\phi(\alpha(a))[\alpha_V(x),\alpha_V(y)]_V\\
\ &&=[x,\phi(a)(\alpha_V(y))]_V+[\phi(a)(\alpha_V(x)),y]_V.
\end{eqnarray*}
Therefore we proved the converse. Hence the  identities \eqref{plus1} and \eqref{plus2} are equivalent.

We set for $a\in A$ and $i\in I$, $[a,i]:=a\cdot i.$ The Hom-Jacobi condition
$[\alpha(a),[b,i]]+[\alpha(b),[i,a]]+[\alpha(i),[a,b]]=0$ leads to
\begin{equation}\label{doublestar}
\alpha(a)\cdot (b\cdot i)-\alpha(b)\cdot(a\cdot i)=[a,b]\cdot (\alpha (i))
\end{equation}
which is equivalent to \eqref{star2}.

Denote
$
l: \begin{array}{c}
 A\rightarrow End(I) \\
a\mapsto l (a)
\end{array}
$ where
$
l(a): \begin{array}{c}
I \rightarrow I \\
i\mapsto l(a)(i)=[a,i]=a\cdot i
\end{array}
$.

The identity \eqref{doublestar} may be written for $a,b\in A$
\begin{equation}
l([a,b])\circ \alpha_I =l(\alpha_A(a))l(b)-l(\alpha_A(b))l(a).
\end{equation}
Notice that
\begin{equation}
l\circ \alpha_A (a)=\alpha_I\circ l(a) \circ\alpha_I
\end{equation}
since $l(\alpha_A(a))(i)=[\alpha_A (a),i]=\alpha_I [a,\alpha_I (i)].$

Also we have for $i\in I$ and $a,b,c\in A$
$$B([a,b],i)=B(a,[b,i])=B(a,l(b)(i))$$
and
\begin{align*}
B(c,[\alpha(a),[b,i]]+[\alpha(b),[i,a]]+[\alpha(i),[a,b]]&=B([b,[\alpha(a),c]]-[a,[\alpha(b),c]]+\alpha([[a,b],c]),i)\\
\ & =B([b,[\alpha(a),c]]+[a,[c,\alpha(b)]]+\alpha([c,[b,a]]),i)\\
\ & =B(\alpha([\alpha(b),[a,\alpha(c)]]+[\alpha (a),[\alpha(c),b]]+[c,[b,a]]),i)\\
\ & =0
\end{align*}
since $c=\alpha(\alpha(c)).$

It is clear that $
\nabla: \begin{array}{c}
I \rightarrow A^\ast \\
i\mapsto \nabla(i)=B( i,\cdot )
\end{array}
$ is an isomorphism of vector spaces.

Denote $
\alpha_A^\ast: \begin{array}{c}
A^\ast \rightarrow A^\ast \\
f\mapsto \alpha_A^\ast (f)=f\circ\alpha_A
\end{array}
$ i.e. $\alpha_A^\ast= ^t \alpha_A.$

Let $i\in I$ and $a\in A$, we have
$ \nabla(\alpha_I (i)(a)=B(\alpha_I(i),a)=B(i,\alpha_A(a))=\nabla(i)\circ\alpha_A(a).$
Then
\begin{eqnarray}\alpha^\ast (\nabla(i))=\nabla(i)\circ\alpha_A=\nabla(\alpha_I (i).
\end{eqnarray}
Also $B(a\cdot i,b)=B([a,i],b)=-B(i,[a,b])$ that is
\begin{equation}
\nabla (a\cdot i)(b)=-\nabla (i)([a,b]).
\end{equation}
 Thus $\nabla$ is an isomorphism of the $A$-modules $I$ and $A^\ast.$

Recall that if $(\g,[\ ,\ ],\alpha )$ is a Hom-Lie algebra then $( \g \ast^ , \widetilde{ad},\widetilde{\alpha })$ is a representation if and only if  for $x,y,z\in \g$
\begin{equation}\label{cd}
\alpha([[x,y],z])=[y,[\alpha(x),z]]+[x,[z,\alpha(y)]].
\end{equation}
In our case $\alpha$ is an involution which leads to $\alpha_A$ is also an involution. Then the identity \eqref{cd} is satisfied  which implies that $(A^\ast , \widetilde{ad_A},\alpha_A^\ast= ^t \alpha_A)$ is a representation of $(A, [\ ,\ ]_A,\alpha_A)$ (i.e. $A^\ast$ is an $A$-module).
 $$
\phi: \begin{array}{c}
V\times V \rightarrow I  \rightarrow A^\ast \\
(x,y) \mapsto \varphi(x,y) \mapsto B(\varphi(x,y), \cdot )
\end{array}
$$
Consider $\Gamma :\g=A\oplus V\oplus I\rightarrow A\oplus V\oplus A^\star$ defined by $\Gamma(a+v+i)=a+v+\nabla (i)$. It is easy to show that it is a Hom-Lie algebras morphism. Notice that  $A\oplus V\oplus A^\star$ is a double extension of $V$ by $A$.
\end{proof}
Let us denote by $\mathcal{U}$ the set incorporating the trivial Hom-Lie algebra $\{0\}$, the one-dimensional Hom-Lie algebra and  simple involutive Hom-Lie algebras.
 \begin{theorem}
 Let  $( \g, [\ ,\ ],\alpha,B)$ be an    IQH-Lie algebra such that $\mathcal{Z}(\g)=\{0\}$. Then  it is either an element of $\g$ or it is obtained from a finite number of element of $\mathcal{U}$ by a sequence of involutive  double extension by the one-dimensional Hom-Lie algebra and/or double extension by a simple involutive  Hom-Lie algebra and/or orthogonal direct sum of IQH-Lie algebras.
 \end{theorem}
According to Theorem  \ref{SimpleInvolutiveHomLie} and Proposition \ref{Prop6.3} simple involutive Hom-Lie algebras are determined by simple or semisimple Lie algebras with involution.
\bibliographystyle{amsplain}

\begin{thebibliography}{10}


\bibitem{AizawaSaito}
Aizawa N., Sato H.: \emph{ $q$-Deformation of the Virasoro algebra with central extension,}
Physics Letters B, Phys. Lett. B
\textbf{256}, no. 1, 185--190 (1991).
Aizawa, N., Sato, H., Hiroshima University preprint
preprint HUPD-9012 (1990).

\bibitem{AEM}Ammar F., Ejbehi   Z.  and Makhlouf A.,
\emph{ Cohomology and Deformations of Hom-algebras, }
 arXiv:1005.0456 (2010).
\bibitem {AmmarMakhloufJA2010}Ammar F.   and Makhlouf A., \emph{Hom-Lie algebras and Hom-Lie admissible superalgebras}, Journal of Algebra, Vol. \textbf{324} (7), (2010)  1513--1528.




 \bibitem{AM2008} Ataguema H., Makhlouf A. and Silvestrov  S., \emph{
Generalization of $n$-ary Nambu algebras and beyond}, Journal of
Mathematical Physics \textbf{50}, 1 (2009).

\bibitem{BajoBenayadi} Bajo I.  and Benayadi S. , \emph{Lie algebras with quadratic dimension equal to 2}, Journal of Pure
and Applied Algebra \textbf{209}, no. 3,  (2007) 725 --737.

\bibitem{BenamorBenayadi}  Benamor H.  and  Benayadi S., \emph{Double extension of quadratic Lie superalgebras}, Communications in
Algebra, 27 (1), (1999)  67--88.

\bibitem{BenayadiJA2000} Benayadi  S., \emph{Quadratic Lie superalgebras with the completely reducible action of even part on
the odd part,} Journal of Algebra., \textbf{223}, (2000) 344--366.

\bibitem{BenayadiJA2003} \bysame \emph{Socle and some invariants of Quadratic Lie superalgebrast,} Journal of Algebra., \textbf{261}, (2003) 245--291.

\bibitem{Bordemann97} Bordemann M., \emph{NonDegenerate invariant bilinear
forms in nonassociative algebras,} Acta Math. Univ. Comenian. LXVI (2) (1997),
 151--201.
\bibitem{Canepl2009}  Caenepeel S.,   Goyvaerts I., \emph{Monoidal
Hom-Hopf algebras}, arXiv:0907.0187v1 [math.RA]  (2009).

\bibitem{ChaiElinPop} Chaichian M., Ellinas D., Popowicz Z., \emph{ Quantum conformal algebra with
central extension,} Phys. Lett. B \textbf{248}, no. 1-2, (1990) 95--99.

\bibitem{ChaiKuLukPopPresn} Chaichian, M.,
Isaev A. P., Lukierski J., Popowicz Z.,
Pre\v{s}najder P., \emph{ $q$-Deformations of Virasoro
algebra and conformal dimensions,} Phys. Lett. B
\textbf{262}   (1), (1991) 32--38.

\bibitem{ChaiIsKuLuk}  Chaichian M., Kulish P.,
Lukierski J., \emph{ $q$-Deformed Jacobi identity,
$q$-oscillators and $q$-deformed
infinite-dimensional algebras,} Phys. Lett. B
\textbf{237} , no. 3-4, (1990) 401--406.

\bibitem{ChaiPopPres} Chaichian, M.,
Popowicz, Z. and Pre\v{s}najder, P., \emph{  $q$-Virasoro
algebra and its relation to the $q$-deformed KdV
system,} Phys. Lett. B \textbf{249}, no. 1,
 (1990) 63--65.

\bibitem{CurtrZachos1} Curtright T. L., Zachos C. K., \emph{  Deforming maps for quantum algebras,} Phys. Lett.
B \textbf{243}, no. 3, 237--244  (1990).

\bibitem{DaskaloyannisGendefVir}
Daskaloyannis, C., \emph{ Generalized deformed Virasoro algebras}, Modern
Phys. Lett. A \textbf{7} no. 9, (1992) 809--816.


\bibitem{Dixmier2} Dixmier J. \emph{Enveloping algebras}, Graduate studies in Math, 11\textbf{11}, AMS, 1996.

\bibitem{FregierGohrSilv} Fregier Y.,  Gohr A. and Silvestrov S., \emph{Unital algebras of Hom-associative type and surjective or injective
twistings}, J. Gen. Lie Theory Appl. Vol. \textbf{3} (4), (2009) 285-295.

\bibitem{Gohr} Gohr A. , \emph{On Hom-algebras with surjective
twisting}, arXiv:0906.3270v3 [Math.RA] (2009).


\bibitem{HLS} Hartwig J. T., Larsson D., Silvestrov S. D., \emph{ Deformations of Lie algebras using $\sigma$-derivations}, J. of
Algebra \textbf{295}, (2006)  314-361.

\bibitem{HS} Hellstr\"{o}m L., Silvestrov S. D., \emph{  Commuting elements in $q$-deformed Heisenberg algebras}, World
Scientific (2000).

\bibitem{Hu}  Hu N., \emph{  $q$-Witt algebras,
$q$-Lie algebras, $q$-holomorph structure and
representations,}  Algebra Colloq. {\bf 6} ,
no. 1, (1999) 51--70.

 \bibitem{JinLi} Jin Q. and  Li X., \emph{Hom-Lie algebra structures on semi-simple Lie algebras},
Journal of Algebra, Volume \textbf{319}, Issue 4, (2008) 1398--1408


\bibitem{Kassel1} Kassel, C., \emph{Cyclic homology of differential operators,
the Virasoro algebra and a $q$-analogue}, Commun. Math. Phys. 146
(1992), 343--351.




\bibitem{LS1} Larsson D., Silvestrov S. D., \emph{  Quasi-Hom-Lie algebras, Central Extensions and 2-cocycle-like
identities,} J.  of Algebra \textbf{288},  (2005) 321--344.

\bibitem{LS2} Larsson D., Silvestrov S. D., \emph{  Quasi-Lie algebras,} in {\it Noncommutative
Geometry and Representation Theory in Mathematical Physics,}
Contemp. Math., \textbf{391}, Amer. Math. Soc., Providence, RI,  (2005)  241--248.

\bibitem{LS3} Larsson D., Silvestrov S. D., \emph{  Quasi-deformations of $sl_2(\mathbb{F})$
using twisted derivations}, Comm. in Algebra \textbf{35}, (2007) 4303 -- 4318.

\bibitem{Lister} Lister W.G., \emph{A structure theory of Lie triple systems}, Trans. Amer. Math. Soc. \textbf{72}, (1952) 217--242.


\bibitem{LiuKeQin} Liu, Ke Qin,
\emph{Characterizations of the quantum Witt algebra}, Lett. Math. Phys.
\textbf{24} , no. 4, (1992)  257--265.




\bibitem{MS} Makhlouf A. and Silvestrov S. D.,
\emph{Hom-algebra structures}, J. Gen. Lie Theory Appl. \textbf{2}
(2) , (2008) 51--64.

\bibitem{HomHopf} \bysame
\emph{Hom-Lie admissible Hom-coalgebras and Hom-Hopf algebras},
 Published  as Chapter 17, pp
189-206, {\rm S. Silvestrov, E. Paal, V. Abramov, A. Stolin, (Eds.),
Generalized Lie theory in Mathematics, Physics and Beyond,
Springer-Verlag, Berlin, Heidelberg, (2008).}

\bibitem{HomDeform} \bysame
\emph{Notes on Formal deformations of Hom-Associative and Hom-Lie
algebras},  Forum Mathematicum, vol. \textbf{22} (4) (2010), 715--759.

\bibitem{HomAlgHomCoalg} \bysame
\emph{Hom-Algebras and Hom-Coalgebras},  Journal of Algebra and its Applications, Vol. \textbf{9} , DOI:10.1142/S0219498810004117, (2010).

\bibitem{MedinaRevoy} Medina A. and Revoy Ph., \emph{Alg\`{e}bres de Lie et produit
scalaire invariant}, Ann. Sci. Ecole Norm. Sup. \textbf{4} 18, (1985) 553--561.

\bibitem{Pinczon10}  Minh Thanh D., Pinczon G., Ushirobira R.,\emph{ A new invariant of quadratic Lie algebras}, 	arXiv:1005.3970v2 (2010).



\bibitem{Sheng} Sheng Y., \emph{Representations of hom-Lie algebras}, 	arXiv:1005.0140v1 [math-ph] (2010).


\bibitem{taft} Taft E., \emph{Invariant Levi  factors}, Michigan Math. J. \textbf{9} (1962) 65--68.
\bibitem{TauvelLivre} Tauvel P. and Yu R.W.T., \emph{Lie algebras and algebraic groups}, Springer Verlag, Berlin Heidelberg, 2005.


\bibitem{Yau:EnvLieAlg} Yau D.,
\emph{Enveloping algebra of Hom-Lie algebras}, J. Gen. Lie Theory Appl.
\textbf{2} (2) (2008), 95--108.
\bibitem{Yau:HomolHomLie} \bysame
 \emph{Hom-algebras as deformations and homology},
arXiv:0712.3515v1 [math.RA] (2007).
\bibitem{Yau:HomBial} \bysame
 \emph{Hom-bialgebras and comodule algebras},
arXiv:0810.4866v1 (2008).
 \bibitem{Yau4}\bysame
\emph{ Hom-Malsev, Hom-alternative, and Hom-Jordan algebras,}  arXiv 1002.3944 (2010).
\end{thebibliography}
\providecommand{\bysame}{\leavevmode\hbox to3em{\hrulefill}\thinspace}
\providecommand{\MR}{\relax\ifhmode\unskip\space\fi MR }
\providecommand{\MRhref}[2]{%
  \href{http://www.ams.org/mathscinet-getitem?mr=#1}{#2}
}
\providecommand{\href}[2]{#2}

\end{document}